\documentclass{amsart}
\usepackage[margin=1in]{geometry}
\usepackage{amssymb}
 
\newcommand{\nball}{\mathbb{B}^n}

\newcommand{\dup}[2]{\langle#1,#2 \rangle}

\newcommand{\ip}[2]{(#1,#2)}
\newcommand{\Co}{\mathrm{Co}} 
\newcommand{\cdual}{{\ast}}

\newcommand{\sun}{\mathrm{SU}(n,1)}

\newcommand{\fS}{\mathcal{S}} 

\newtheorem{theorem}{Theorem}[section]
\newtheorem{lemma}[theorem]{Lemma}
\newtheorem{corollary}[theorem]{Corollary}
\newtheorem{proposition}[theorem]{Proposition}
\theoremstyle{definition}

\newtheorem{remark}[theorem]{Remark}

 \newcommand{\note}[2][\null]{%
   \marginpar{\renewcommand{\baselinestretch}{1}\vspace{-1em}\hrule\vspace{3pt}%
  \raggedright\textsf{ \footnotesize#2\ifx \footnotesize#1\null\else\\\hfill--- 
   {\em #1}\fi}\vspace{1.5em}}%
 }

\newcommand{\wf}{\widetilde{f}}
\newcommand{\wg}{\widetilde{g}}
\newcommand{\cS}{\mathcal{S}}
\newcommand{\cA}{\mathcal{A}}
\newcommand{\cH}{\mathcal{H}}
\newcommand{\cP}{\mathcal{P}}
\newcommand{\cV}{\mathcal{V}}
\newcommand{\Z}{\mathbb{Z}}
\newcommand{\fg}{\mathfrak{g}}

\newcommand{\Ad}{\mathrm{Ad}}

\renewcommand{\note}{\footnote}

\begin{document}

\title{New atomic decompositons for Bergman spaces on the unit ball} 
\subjclass[2010]{Primary
  32A36, 43A15, 42B35, 46E15; Secondary 22D12} 
\keywords{Coorbit spaces,  Bergman spaces, representation theory of
  Lie groups, sampling theory} 

\author{Jens Gerlach Christensen} 
\address{ Department of Mathematics, Colgate University} 
\email{jchristensen@colgate.edu}
\urladdr{http://www.math.colgate.edu/~jchristensen}
\thanks{The research of J. Christensen was partially supported by NSF grant DMS-1101337.}
\author{Karlheinz Gr\"ochenig} 
\address{Faculty of Mathematics, University of Vienna } 
\email{karlheinz.groechenig@univie.ac.at}
\urladdr{http://homepage.univie.ac.at/karlheinz.groechenig/}
\thanks{K.\ Gr\"ochenig  acknowledges the partial 
  support of  the  project P26273 - N25  of the
Austrian Science Fund (FWF) and the great hospitality of CIRM,
Marseille, during the completion of this work.}
\author{Gestur \'Olafsson} 
\address{ Department of Mathematics, Louisiana State University} 
\email{olafsson@math.lsu.edu}
\urladdr{http://www.math.lsu.edu/~olafsson}
\thanks{The research of G. \'Olafsson was supported by NSF grant DMS-1101337.}

\begin{abstract}
  We derive atomic decompositions and frames for 
  weighted Bergman spaces of several complex variables on the unit
  ball in the spirit of  Coifman, Rochberg,
  and Luecking. In contrast to our predecessors,   we use group theoretic methods, in particular the
  representation theory of the discrete series of $\sun $ and its covering
  groups.  One of the benefits is a much larger class of admissible
  atoms. 
\end{abstract}

\maketitle

\section{Introduction}
\label{sec:intro}

\noindent
Coorbit theory is an abstract theory for the construction of atomic
decompositions and frames in the presence of some group
invariance~\cite{Christensen2012,Christensen2011,Christensen2013,Feichtinger1988,Feichtinger1989a,Fornasier05,
  Fuehr2014,Fuehr2014a,Grochenig1991}.  
The input data for coorbit theory are (i)  a representation $(\pi,\cS)$ of a locally compact group $G$ in a Fr\'echet space $\cS$ embedded continuously and densely in
it conjugate dual $\cS^*$, (ii) a  cyclic vector $\varphi \in \cS$, and
(iii)  a $G$-invariant function space $Y$ on $G$ (usually $Y$ is a weighted $L^p$-space
on $G$). Under suitable additional assumptions  on $(\pi ,\psi , Y)$ one can then
define a Banach space $\Co _\cS ^\psi Y$, the \emph{coorbit of $Y$},
via representation coefficients: a distribution 
$f\in \cS^*$ is in $\Co _\cS ^\psi Y$, if the corresponding representation
coefficient $x \to \langle f, \pi (x)\psi\rangle $ belongs to  $Y$.  The
main achievement of coorbit theory is to derive the existence of
atomic decompositions and Banach frames for the coorbit spaces. 
The various levels of generality and the assumptions required to derive
such results have become an active topic of research, and currently
several generalizations of the original theory exist and are applied
to new examples, see~\cite{Christensen2011,Christensen2012a,DKST09,Fuhr2012}.

The first set of  examples in~\cite{Feichtinger1988} revealed several classical
families of function spaces as coorbit spaces with respect to some
basic unitary representations in harmonic analysis: (i) For the Schr\"odinger representation of the
Heisenberg group, the coorbit spaces coincide with the  modulation
spaces, and  their 
corresponding frames and decompositions are known as Gabor frames and
Gabor expansions~\cite{book,fg92chui}. (ii) For the quasi-regular representation of
the group of affine translations on $\mathbb{R}$, the coorbit spaces
are the well-known  Besov
spaces, and the corresponding expansions and frames are the  wavelet
expansions and frames~\cite{Feichtinger1988,RU11}. (iii) For the discrete series of
$\mathrm{SL}(2,\mathbb{R})$ the coorbit spaces are  certain Bergmann spaces on the
unit disc, and coorbit theory yields  their atomic decompositions in
the sense of Coifman and Rochberg~\cite{Christensen2009,Feichtinger1988}.
 
{}From this point of view, coorbit theory represents a unifying
theory that reveals the common background of these three branches of
analysis. Clearly these examples can be  studied with 
methods specifically  tailored to these function spaces, and often
sharper results can be proved.  However,
coorbit theory did not only explain and rederive 
known results; the generality of the set-up also lead to new results
about these known classes of function spaces. For example, coorbit
theory furnished non-uniform Gabor and wavelet expansions long before
the study of non-uniform expansions became a topic of research in
function spaces.   The coorbit theory of these examples has been worked
out in fine detail, and each example now stands for an
independent direction of research in coorbit theory. 

Nowadays,  coorbit theory is also  used to define and investigate new
function spaces. The coorbit spaces attached to  certain groups of affine transformations on
$\mathbb{R}^d$, so-called shearlet groups, are new function spaces
(\emph{shearlet spaces}) for which no previous information is
known. In these cases, coorbit theory provides the initial tools for
their analysis~\cite{Dahlke2012}. 

This paper is the first attempt to develop the coorbit theory for the
discrete series of semisimple Lie groups. Motivated by the  example
$\mathrm{SL}(2,\mathbb{R})$ and Bergmann spaces in ~\cite{Christensen2009,Feichtinger1988},
our goal
is to  develop  the coorbit theory attached to the
holomorphic discrete series of the simple Lie groups $\sun $ and their 
covering groups. Realizing $\sun /S(\mathrm{U}(n)\times \mathrm{U}(1))$ as the unit ball $\nball$
in $\mathbb{C}^n$ those  representations act on Hilbert spaces of holomorphic functions on $\nball$.
The first objective is to identify
and describe the coorbit space with concrete and possibly already
known function spaces. The second objective then is to make the
abstract atomic decompositions more concrete and formulate these
explicitly. Although the step from the three-dimensional
group $\mathrm{SL}(2,\mathbb{R} ) \simeq \mathrm{SU}(1,1)$ to the higher-dimensional
groups $\sun $ is natural, the technicalities of the representation theory of semisimple groups and of
analysis in several  complex variables are much more demanding. This
may be the reason why the coorbit theory of semisimple Lie groups has mostly
been untouched territory so far.

To give the reader some idea of the main results, we state  those without insisting in precise definitions and formulations. 
For $\alpha>-1$ the weighted Bergman space
$\mathcal{A}^p_\alpha$ is the space of all holomorphic functions on $\nball$ such that
\[ \|f\|_\alpha^p=c_\alpha \int_{\nball }|f(z)|^p (1-|z|^2)^\alpha\, dv(z)<\infty\]
where $dv$ is the Lebesgue measure on $\nball $ and $c_\alpha$ is a
normalizing constant so  that the constant function 1 has norm one.
 The space $\mathcal{A}^p_\alpha $  is   a Banach space for $1\leq p\leq \infty$ 
and a Hilbert space for   $p=2$.
We let $\sigma =\alpha + n+1$ and write $\cV_\sigma =\cA_{\alpha}^2$ as $\sigma$ will show
up in other contexts.

The group $\sun$ consists of all  $(n+1)\times (n+1)$  matrices that preserve the
sesquilinear  form 
$ z,w \in \mathbb{C} ^{n+1} \to 
-z_{1}\overline{w}_{1} - \dots -
 z_{n}\overline{w}_{n}+ z_{n+1}\overline{w}_{n+1}$.
Usually,  a matrix $x \in \sun$ is written in block form   as 
$  x= \left(\begin{smallmatrix}
    a & b \\
    c^t & d
  \end{smallmatrix}\right)$. 
   As in the case $n=1$, such a
  matrix acts on $\nball $ by fractional transformations. Given a real number
   $\sigma >n$ one  can then define a representation of $\sun
  $ on functions living on $\nball $  by  
\begin{equation}\label{de:rep}
\pi _\sigma (x) f(z) =  
  \frac{1}{(-\dup{z}{b}+\overline{d})^\sigma}
  f\left( \frac{a^*z-\overline{c}}{-\dup{z}{b} +\overline{d}}\right)\,
  . 
\end{equation}
If this action is restricted to the Hilbert space $\mathcal{V}_\sigma
$, then $\pi _\sigma $ becomes an
irreducible
square-integrable unitary representation of $\sun $, if $\sigma$ is an integer, or one of its covering groups otherwise.

Our first result identifies the coorbit spaces with respect to $\pi
_\sigma $ as Bergman spaces on the unit ball. The following  is an
unweighted version of Theorem~\ref{identberg}. 
\begin{theorem}
  Let   $1\leq p< \infty $, $\sigma > 2n\max\{1/p,1-1/p \}$,   and  $\psi=1_{\nball}$.
  The representation coefficient $x \to \langle f, \pi _\sigma (x)
  \psi \rangle  $ is in $L^p(\sun )$,  if and only if $f\in \mathcal{A}^p_{\sigma p/2 -n-1} $. In other words,
  $f$ is in the coorbit of $L^p$, if and only if $f$ belongs to the
  Bergman space $\mathcal{A}^p_{\sigma p/2 -n-1}$ on $\nball $. 
\end{theorem} 

In Theorem~\ref{atberg} we will formulate the atomic decompositions
for general weighted Bergman spaces. An abbreviated version with fewer
assumptions reads as follows. 
\begin{theorem}
Assume that   $1\leq p< \infty $ and $\sigma > 2n\max\{1/p,1-1/p \}$.  Then  there exist
points  $\{w_i: i\in I\} \subseteq \nball$,  
  such that  every $f\in \mathcal{A}^p_{\sigma p/2 -n-1} $ possesses the
  decomposition 
$$
f(z)= \sum _{i\in I} c_i(f) (1-|w_i|^2)^{\sigma /2} (1-\langle z,
w_i\rangle )^{-\sigma }
$$ 
with unconditional convergence in $\mathcal{A}^p_{\sigma p/2-n-1} $. 
The coefficient sequence $c$ is in $\ell ^p $ and can be chosen so
that $\|c\|_p$ is equivalent to the Bergman space norm of $f$. 
\end{theorem}

We also obtain frame expansions for the Bergman spaces, which is related to the existence of 
sampling sequences. We are not aware of any results in this direction except in the special
case of the unit disc \cite{feipap,pap3,pap2}, though they are 
hardly surprising considering the work by \cite{Seip2004,Duren2004}. 

The precise formulations and proofs  require quite a bit of set-up. An important
technical point is the integrability or lack thereof   of the
representations $\pi _\sigma$ for small values of $\sigma $. 
In fact, whereas  all
these representations 
are square integrable, for some of them there is a $p'$ such that
the matrix coefficients are only in $L^p$ for $p>p'>1$. In \cite{Feichtinger1989a}, and more
generally \cite{Dahlke04,Fornasier05}, integrability was crucial for
the derivation and discretization of a reproducing formula. However, this  
restriction is not necessary, as is shown in
\cite{Christensen2009,Christensen2011,Christensen2012,FG07}.
In a
sense, the concrete examples of Bergmann spaces justify and support
the considerable efforts spent to extend the scope of coorbit theory
to non-integrable and reducible representations. 

It is a curious fact that Zhu's book on Bergman spaces  does not mention $SU(n,1)$
and its representations (though the symmetry group of the unit ball
in $\mathbb{C}^n$ makes a guest appearance). In our approach, $\sun $
plays the main role for the understanding of Bergman spaces and their
atomic decompositions. It allows us to use general methods to derive specific results
for this example and guides the way for similar results
for all symmetric bounded domains. 

This paper is organized as follows: In Section~2 we recall the
definition of the Bergman spaces and the  basic
facts about $\sun $ and its discrete series. A basic lemma
(Lemma~\ref{le:1andK}) establishes the connection between the reproducing
kernel of the Bergman space and the action of the discrete series of
$\sun $ on a ``distinguished'' vector. In Section~3 we develop the
coorbit theory for the discrete series $\pi _\sigma $ of $\sun $. First we recall the
general definitions and concepts of coorbit theory. In Theorem~\ref{identberg}
we show that the coorbits with respect to  the representations $\pi
_\sigma $ are precisely  the weighted Bergman spaces  of several
complex variables on the unit ball. With this identification, the
atomic decompositions of Bergman spaces and several sampling formulas
can then be derived directly (though not without some effort) from the
abstract results in coorbit theory. 

\section{Bergman spaces on the unit ball}
\noindent
In this section we prepare the background for the complex analysis of
several variables on the unit ball and for the related representation
theory of Lie groups. 

\subsection{Bergman spaces on the unit ball}
The space $\mathbb{C}^n$ is equipped with the inner product
$\dup{w}{z} = w\cdot\overline{z}
=w_1\overline{z}_1+\dots+w_n\overline{z}_n$,  and
the open unit ball $\nball$ in $\mathbb{C}^n$ consists of 
$n$-tuples $z=(z_1,\dots,z_n)^t$
for which $|z|^2 = |z_1|^2+\dots +|z_n|^2 < 1$.
The ball  $\nball$ can be identified with the unit
ball in $\mathbb{R}^{2n}$ and thus can be equipped with the measure
$dv = 2nr^{2n-1}\,dr\,d\sigma_n$ where $d\sigma_n$ is the rotation-invariant surface
measure on the sphere $\mathbb{S}^{2n-1}$ in $\mathbb{R}^{2n}$ normalized by $\sigma_n (\mathbb{S}^{2n-1})=1$. Then
$dv$ is a rotation-invariant probability measure on $\nball$.

For $\alpha>-1$ define the probability 
measure $dv_\alpha(z) = c_\alpha (1-|z|^2)^\alpha dv(z)$ where
$c_\alpha = \frac{\Gamma(n+\alpha+1)}{n!\Gamma(\alpha+1)}$,
and let $L_\alpha^p(\nball)$ be the weighted Lebesgue space 
carrying the norm
\begin{equation*}
  \| f\|_{L^p_\alpha} = \left( \int |f(z)|^p \,dv_\alpha(z)\right)^{1/p}.
\end{equation*}
For $\alpha>-1$ the weighted Bergman space
$\mathcal{A}^p_\alpha$ is the space of all holomorphic functions on $\nball$ which are
in $L^p_\alpha(\nball)$. The space $\mathcal{A}^p_\alpha $  is closed in   $L^p_\alpha$ and
hence a Banach space for $1\leq p<\infty$, and
the space of (holomorphic) polynomials is dense in $\mathcal{A}^p_\alpha (\nball )$.

When $p=2$ the space $\mathcal{A}^2_\alpha$ is a Hilbert space
with the  inner product
\begin{equation*}
  \ip{f}{g}_\alpha = \int_{\nball} f(z)\overline{g(z)}\,dv_\alpha(z).
\end{equation*}
For a multiindex $\gamma =(\gamma_1, \ldots , \gamma_n)\in \Z_+^n$ we write
$|\gamma |=\gamma_1+\ldots +\gamma_n$ and $z^\gamma=z_1^{\gamma_1}
\cdots z_n^{\gamma_n}$. Then a simple calculation, see \cite[Lemma 1.11]{Zhu2005}, shows that
\begin{equation}\label{eq:nonZgamma}
\|z^\gamma \|_{ \alpha}^2=\frac{\gamma! \Gamma (n+\alpha+1)}{\Gamma (n+\alpha +1+|\gamma |)}\, .
\end{equation}
From this and the fact that the polynomials are dense in $\cA^2_\alpha$ one shows the
following lemma:
\begin{lemma}\label{le:repK} The functions
\[\varphi_\gamma (z)=\left(\frac{\Gamma (n+\alpha+1+|\gamma|)}{\gamma ! \Gamma (n+\alpha  +1)}\right)^{1/2}\, z^\gamma\, ,\quad
\gamma\in\Z_+^n\, \]
form a orthonormal basis for $\cA_\alpha^2$. Furthermore, if
\begin{equation}\label{eq:rk}
K_ \alpha(z,w) = \frac{1}{(1-\dup{z}{w})^{n+1+\alpha}}=\sum_{\gamma } \varphi_\gamma (z)\overline{\varphi_\gamma (w)},
\end{equation}
then
\[P_\alpha f(z)=\int_{\nball} f(w)K_\alpha (z,w)\, d\mu_\alpha (w)\]
is the orthogonal projection from $L^2_\alpha(\nball)$ onto $\cA_\alpha^2$. In particular, for $f\in \cA_\alpha^2$  we have
\[f(z)=\int_{\nball} f(w)K_\alpha (z,w)\, d\mu_\alpha (w)\, .\]
Thus $K_\alpha (z,w)$ is the reproducing kernel for $\cA_\alpha^2$.
\end{lemma} 
Note that $\overline{K_\alpha (z,w)}=K_\alpha (w,z)$.
If $w\in\nball$, then
$|K_\alpha (z,w)|\le (1-|w|^2)^{-(n+1+\alpha)}<\infty$ for  all  $z\in \nball$,
and therefore the function 
$z\mapsto K_\alpha (z,w)$ is in all of the spaces  $\cA^p_\alpha$, 
$1\leq p < \infty$.
%

\subsection{Groups acting  on the unit ball and integral formulas}

The group $\sun$ consists of all  matrices with determinant one that preserve the
sesquilinear  form 
\[ \dup{J_{n,1}z}{w}
=  - z_{1}\overline{w}_{1} - \dots - z_{n}\overline{w}_{n}+z_{n+1}\overline{w}_{n+1}\, ,\]
where 
$$
J_{n,1} = \begin{pmatrix}
    -I_n & 0 \\
    0 & 1
  \end{pmatrix}.
$$ 
We usually write a matrix  $x$ in $\sun$ in block form as 
\begin{equation*}
  x= \begin{pmatrix}
    a & b \\
    c^t & d
  \end{pmatrix},
\end{equation*}
where $a$ is an $n\times n$-matrix with complex entries,  $b,c\in
\mathbb{C}^n$, and  $d\in \mathbb{C}$. Then by definition  $x$
is in $\sun$ if and only if $\det x =1$ and
$x^*J_{n,1}x=J_{n,1}$.
This immediately gives that 
\begin{equation}\label{eq:inverse}
x^{-1} = J_{n,1}x^*J_{n,1} = \begin{pmatrix}
    a^* & -\overline{c} \\
    -\overline{b}^t & \overline{d}
  \end{pmatrix}.
\end{equation}
The group $\sun$ acts transitively on $\nball$ via the fractional linear transformations
\begin{equation*}
  \begin{pmatrix}
    a & b \\
    c^t & d
  \end{pmatrix}\cdot z = \frac{az+b}{\dup{c}{\overline{z}} + d},
\end{equation*}
We denote the origin in $\mathbb{C}^n$ by $o$.

Next define the subgroups $K$, $A$,  and $N$ of $\sun$ by
\begin{align*}
  K &= \left\{ \left.  u_k=\begin{pmatrix}
      k & 0 \\
      0 & \overline{\det(k)}
    \end{pmatrix} \,\right|\,  k \in U(n)\right\}\simeq U(n) \\
  A &= 
  \left\{  \left. a_t= \begin{pmatrix}
      \cosh(t) & 0 & \sinh(t)\\
      0 & I_{n-1} & 0 \\
      \sinh(t) & 0 &\cosh(t)
    \end{pmatrix} \,\right|\, t\in\mathbb{R}\right\}\simeq \mathbb R \\
  N &= 
  \left\{\left. n_{z,s} = \begin{pmatrix}
      1-|z|^2/2 + is & z^T & |z|^2/2-is\\
      -\overline{z} & I_{n-1} & \overline{z} \\
      is-|z|^2/2 & z^T & 1+|z|^2/2-is
    \end{pmatrix} \,\right|\, z\in \mathbb{C}^{n-1},\, s\in\mathbb{R}\right\} \\
\end{align*}
{}From now on we let $S=AN$ and write $G=\sun$. The following is well known
\cite{Helgason1978,Knapp1986}:
\begin{theorem}[Two Integral Formulas] \label{th:2.3} The following holds:
\begin{enumerate}
\item The map $A\times N\times K \to G$, $(a,n,k)\mapsto ank$, is an
analytic diffeomorphism.
\label{th:2.3-1}
\item The Haar measures on $K, A, N$  can be  normalized such that for all $f\in L^1(G)$
\[\int_G f(x)\, dx=\int_A\int_N\int_K f(ank)\, dkdnda=\int_S\int_K f(hk)\, dkdh\, .\]
\label{th:2.3-2}
\item Both $G$ and $S$ act transitively on $\nball$. 
The stabilizer of $o$ in $G$ is
$K$ and  $S$ acts freely. Thus the mapping $h\mapsto h\cdot o$ defines  equivariant 
maps from $S $ and $G$ onto $\nball$ and yields the diffeomorphisms 
\[\nball \simeq G/K \simeq S\, .\]
\label{th:2.3-3}
\item The group $S$ is simply connected and solvable.
\label{th:2.3-4}
\end{enumerate}
\end{theorem}

Right $K$-invariant functions on $G$ (or, equivalently,  functions on $S$) correspond bijectively to
functions on $\nball$   via the identification 
\[\widetilde{f}(x)=f(x\cdot o)\, .\]
This identification induces a  $\sun$-invariant measure   on
$\nball$, which, up to a multiplicative  constant, 
is given by
\begin{equation*}
  f\mapsto \int_{\nball} f(z) \, \frac{dv(z)}{(1-|z|^2)^{n+1}}\, .
\end{equation*}
 We normalize this measure  so that
\begin{equation}\label{eq:int}
\int_{\nball} f(z)\,  \frac{dv(z)}{(1-|z|^2)^{n+1}}=\int_G \wf(x)\, dx=\int_S \wf(x)\, dx\, .
\end{equation}
We then obtain the following isometry between  $L^p$-spaces. 
\begin{lemma} Let $f\in L^p_\alpha (\nball)$ and let $\widetilde{f}(x)=f(x\cdot o)$. Then
\begin{equation*}
 \label{eq:normequiv}
  \| f\|_{L^p_\alpha(\nball)} ^p
  :=    \int_{\nball} |f(z)|^p \,dv_\alpha(z) 
  = c_\alpha \int_G  |\widetilde{f}(x)|^p 
    (1-|x\cdot o|^2)^{\alpha+n+1}\,dx  := c_\alpha 
  \| \wf \|_{L^p_{\alpha+n+1}(G)}^p\,  
\end{equation*}
and similarly for the group $S$.
\end{lemma}

The integration formula (\ref{eq:int}) shows that for suitable functions the convolution on $G$
respectively $S$ can be written as integration over the ball. For
further reference, we state
this fact as a lemma. 
\begin{lemma} \label{conv} Let $f$ and $g$ be functions on $\nball$ and $\wf$ and
  $\wg$ be the corresponding  functions on either $G$ 
or $S$. Let
$\wg^\vee (x)=\wg(x ^{-1})$. Then the
following holds:
\begin{enumerate}
\item  Assume that $\wf$ is right $K$-invariant and $\wg$ is
  $K$-biinvariant on $G$. Then
\[(\wf*\wg ) (x)=\int_{\nball} f(z)g^\vee (x^{-1}z)\frac{dv
  (z)}{(1-|z|^2)^{n+1}}\, ,\]
whenever the convolution on $G$ is defined. 
\item In the case of $S$ we have
\[(\wf*\wg )(x)=\int_{\nball} f(z)g^\vee (x^{-1}z)\frac{dv (z)}{(1-|z|^2)^{n+1}}\, .\]
\end{enumerate}
\end{lemma}

\subsection{Bergman spaces and representations}

In this section we introduce the holomorphic discrete series representations
of $\mathrm{SU}(n,1)$ and its coverings. We also discuss the smooth vectors and the distribution vectors 
of those representations. This is needed to realize the Bergman spaces
as coorbits.  We refer to \cite{FK90,HC55,NW79} for details and proofs. Readers with basic knowledge of
representation theory can skip this subsection and use it as a reference as needed later on.

The action of $\sun$ on $\nball$ yields, for each 
$\sigma >n$,  an irreducible unitary (projective) 
representation of $\sun$
on the Hilbert space $\mathcal{V}_\sigma=\mathcal{A}^2_{\sigma-n-1}$
by 
\begin{equation}\label{de:rep}
\pi _\sigma (x) f(z) =   \pi_\sigma
  \begin{pmatrix}
    a & b \\
    c^t & d
  \end{pmatrix}
  f(z) = \frac{1}{(-\dup{z}{b}+\overline{d})^\sigma}
  f\left( \frac{a^*z-\overline{c}}{-\dup{z}{b} +\overline{d}}\right)\, .
\end{equation}
When $\sigma$ is not an integer, then, because of the factor
$(-\dup{z}{b}+\overline{d})^\sigma$, $\pi_{\sigma}$ is no longer a 
representation of $G$, but it  can  always be lifted to  a representation of the universal covering
$\widetilde{G}$ of $G$ because the pullback of $K$ in $\widetilde G$ is simply connected.
If $\sigma$ is rational, then there exists a finite  covering $G_\sigma$ of $G$ such that
$\pi_\sigma$ is a well defined representation of $G_\sigma$. If $\sigma$ is irrational, then we set
$G_\sigma =\widetilde G$. Then $G_\sigma$ has finite center $Z_\sigma$ and the pullback $K_\sigma$ of
$K$ in $G_\sigma$ is compact, if and only if $\sigma$ is rational.
However, as the subgroup $S$ is simply connected and hence isomorphic to its pullback in $\widetilde G$,
the restriction $\left.\pi_\sigma\right|_S$ of $\pi_\sigma$ to $S$ is a well defined
representation of $S$ for all $\sigma>n$.  To simplify the
exposition we will carry out most  computations for $S$ and denote $\left. \pi_\sigma\right|_S$ simply by
$\pi_\sigma$.

 For each $z\in Z_\sigma$ the operator $\pi_\sigma (z)$ commutes with
all $\pi_\sigma (x)$, $x\in G_\sigma$. As $\pi_\sigma$ is irreducible and unitary it follows by
Schur's Lemma that there exists a homomorphism
$\chi_\sigma : Z_\sigma \to \mathbb{T}$ such that $\pi_\sigma
(z)=\chi_\sigma (z)\, \mathrm{id}$. Hence for
$f,g\in \mathcal{V}_\sigma$, $x\in G_\sigma$, and $z\in Z_\sigma$
we get
\[|\ip{f}{\pi_\sigma (xz)g}|= |\overline{\chi_\sigma (z)} \ip{f}{\pi_\sigma (x)g}|
=|\ip{f}{\pi_\sigma (x)g}|\, .\]
It follows that the absolute value of the wavelet coefficient 
$$W_g^\sigma(f)(x)=\ip{f}{\pi_\sigma (x)g}$$
is in
fact a well defined function on $G_\sigma/Z_\sigma$ and hence a well defined function on $\sun$.

\begin{theorem}\label{th:2.7} 
Let $\sigma >n$. Then 
\begin{enumerate}
\item Let  $f,g\in \mathcal{V}_\sigma$. 
  Then $|W_g^\sigma(f)|$ is square integrable
  on $\sun$. 
\item 
$|W_g^\sigma(g)|$ is
  integrable 
  on $SU(n,1)$ for all   $g$ in a dense $G_\sigma$-invariant 
  subspace of $\mathcal{V}_\sigma$,  if and only if $\sigma >2n$.
  \end{enumerate}
 \end{theorem}

The following lemma reveals the fundamental
connection between the representation theory of $SU(n,1)$ and the
sampling theory in Bergman spaces.  
We fix the notation $\psi=1_{\nball}$, the characteristic function of the ball.
\begin{lemma}\label{le:1andK}
Let $x=\begin{pmatrix}
    a & b \\
    c^t & d
  \end{pmatrix}\in \sun$ and $f\in\cV_\sigma$. Define
  $w=x\cdot o=b/d$. Then for $\psi=1_{\nball}$
\[  \pi_\sigma (x)
    \psi (z) =\overline{d}^{-\sigma} K_{\sigma -n -1}(z,w)
\quad \text{    and }\quad 
W^\sigma\psi (f)(x)=d^{-\sigma}f(w)\, .\]
  \end{lemma}
  
  \begin{proof} The first statements follows from \eqref{eq:rk} and (\ref{de:rep})  using
\[ \pi_\sigma (x) 
    \psi (z) = \frac{1}{(-\dup{z}{b} +\overline{d})^\sigma}
    =\overline{d}^{-\sigma} \frac{1}{(1-\dup{z}{b/d})^\sigma}
    =\overline{d}^{-\sigma} K_{\sigma - n -1} (z,w)\, .\]
The second statement follows then immediately.
\end{proof}

\begin{corollary} \label{rkrep}
  Let $\psi = 1_{\nball }$, $x\in SU(n,1)$,  $w=x \cdot o$, and $f \in
  \mathcal{V}_\sigma $.  Then 
  \begin{equation}
  \label{absvalmatrixcoeff}
  |W_\psi^\sigma(f)(x)| = (1-|w|^2)^{\sigma/2} |f(w)|.
\end{equation}
In particular, for $f=\psi$, we have
\begin{equation}
  \label{eq:c3}
    |W_\psi^\sigma(\psi )(x)| = (1-|x\cdot o|^2)^{\sigma/2} =
    (1-|w|^2)^{\sigma/2} \, .
\end{equation}
\end{corollary}
\begin{proof}
The explicit description of $x^{-1}$ in (\ref{eq:inverse}) shows that
$|d|^2-|b|^2=1$, and therefore
\begin{equation*}
|d|^{-\sigma}=(1-|b/d|^2)^{\sigma/2}=(1-|w|^2)^{\sigma /2}\, .  \qedhere
\end{equation*}
\end{proof}

\begin{proposition} \label{resttos}
 Let $\sigma>n$ and $\psi=1_{\nball}$. Then the following holds true:
\begin{enumerate}
\item  $\psi$ is $S$-cyclic, i.e., the subspace generated by
  the vectors  in  $\{\pi_\sigma (x)\psi : x\in S\}$ is dense in $\mathcal{V}_\sigma$.
\label{resttos-1}

\item $W_\psi^\sigma (f) \in L^2(S)$ for all $f\in \mathcal{V}_\sigma$.
\label{resttos-2}
\item  The space $\{f\in \mathcal{V}_\sigma\mid W_\psi^\sigma (f)\in L^1(S)\}$ is non-zero if and only
if $\sigma >2n$.
\label{resttos-3}
\end{enumerate}
\end{proposition}
\begin{proof} (\ref{resttos-2}) and 
  (\ref{resttos-3}) follow from Theorem \ref{th:2.7} and Theorem
  \ref{th:2.3}(\ref{th:2.3-3}), and  also 
from the explicit  computation for Corollary ~\ref{rkrep}.

To show that $\psi$ is $S$-cyclic in  $V_\sigma$,  it is enough
to show that $W_\psi^\sigma(f)(x)=0$ for all $x\in S$ implies $f=0$.  This immediately follows from
(\ref{absvalmatrixcoeff}) and the fact that $S$ acts transitively on $\nball$.
 \end{proof}

\subsection{Smooth vectors}\label{se:Smooth}
The Fr\'echet space of smooth vectors in $\cV_\sigma$ will play an important role
in the following. We therefore recall the basic definitions and facts here. 
Let $G$ be a Lie group with Lie algebra $\fg$ and let $\pi$
be a representation of   $G$ on a Hilbert space 
$\cH$.  A vector $v\in \cH$ is said to be
\textit{smooth} if the mapping 
$x\mapsto \pi(x)v$, is smooth.
The space of smooth vectors is denoted by $\cH^\infty$. It is dense in $\cH$.
For $X\in\fg$ we define a linear map $\pi^\infty (X) : \cH^\infty \to \cH^\infty $ by
 $$
\pi^\infty(X)v = \lim_{t\to 0} \frac{\pi(\exp(tX))v-v}{t},
$$
where the limit is taken in  $\cH$.  We have
\[\pi (\exp tX)\pi (x)v=\pi(x)\pi (\exp t\Ad (x^{-1})X)\, .\]
It follows therefore, by composition of smooth maps, that $\cH^\infty$ is $G$-invariant. Similarly it follows
that $\cH^\infty $ is $\fg$-invariant. Finally, $\pi^\infty$ is a representation of $\fg$ in $\cH^\infty$ such that
for all $X\in \fg$ and $x\in G$ we have
\[\pi (x) \pi^\infty (X) =\pi^\infty (\Ad (x)X)\pi (x) \, .\]
Recall that for $\mathrm{SU}(n,1)$ the map $\Ad$ is just given by $\Ad (x)X=xXx^{-1}$.

Fix a basis $X_1,\ldots , X_k$ of $\fg$. For $\alpha \in\Z_+^k$ define
\[\pi^\infty (X)^\alpha :=\pi (X_1)^{\alpha_1}\ldots \pi (X_k)^{\alpha_k}\, .\]
We can then define a family of seminorms $p_\alpha$ on $\cH^\infty$ by
\begin{equation}\label{eq:pAlpha}
p_\alpha (v)=\|\pi (X)^\alpha v\|\, .
\end{equation}
This family of seminorms makes $\cH^\infty $ into a Fr\'echet space such that the corresponding representation
of $G$ is continuous, in fact smooth. We denote the space of continuous conjugate
linear functionals on $\cH^\infty$ by
$\cH^{-\infty}$.

The space of smooth vectors in $\cV_\sigma$ can be
described explicitly
as a space of holomorphic functions on $\nball$
with certain decay properties,   which we now describe. This was first done for $\mathrm{SU}(1,1)$ in
\cite{OO88} and then in general in  \cite[p. 551 and p. 556]{Chebli2004}. 
Denote by $\mathcal{P}_k$  the space
of homogeneous polynomials of degree $k$. 
For $\varphi \in \mathcal{V}_\sigma$ write
\[\varphi =\sum_\gamma  (\varphi , \varphi_\gamma ) \varphi_\gamma =
\sum_{k=0}^\infty \sum_{|\gamma  |=k}\ip{\varphi}{\varphi_\gamma }\varphi_\gamma=\sum_{k=0}^\infty \varphi_k\]
where the convergence is in the norm. Thus every element in $\cV_\sigma$ can be written as
a sum of homogeneous polynomials.
\begin{lemma} We have

\begin{equation}\label{eq:VsI}
  \mathcal{V}_\sigma^\infty = \{ 
  \varphi =\sum_k \varphi_k\mid  \varphi_k\in \mathcal{P}_k,
  \text{ and } \forall N,\exists C, \| \varphi_k\|_{\mathcal{V}_\sigma} \leq C (1+k)^{-N}
  \}\, .\end{equation}
and
\begin{equation}\label{eq:VsI1}
  \mathcal{V}_\sigma^{-\infty} = \{ 
  \varphi =\sum_k \varphi_k\mid  \varphi_k\in \mathcal{P}_k,
  \text{ and }  N, C, \| \varphi_k\|_{\mathcal{V}_\sigma} \leq C (1+k)^{N}
  \}\, .\end{equation}
\end{lemma}
Note that the dual pairing can be described as
\begin{equation*}
  \dup{f}{\varphi}_{\mathcal{V}_\sigma}
  = \sum_k (f_k,\varphi_k)
  = \sum_k \int_{\nball} f_k(z)\overline{\varphi_k(z)} 
  (1-|z|^2)^\sigma \frac{dv(z)}{(1-|z|^2)^{n+1}}.
\end{equation*}
We summarize the main properties of $\mathcal{V}_\sigma ^\infty $
needed later. 

\begin{lemma} 
For $\sigma>n$ we have:

\label{bd} 
\begin{enumerate}
\item  Every polynomial is a  smooth vector for $\pi _\sigma $.
\label{bd-1}
\item  $\mathrm{dim} \, (\mathcal{P}_k ) = \binom{k+n-1}{k} \leq (1+k)^n$. 
\label{bd-2}
\item Every smooth vector is bounded on $\nball $. 
\label{bd-3}
\end{enumerate}
  \end{lemma}
  \begin{proof}  Parts    (\ref{bd-1}) and (\ref{bd-2}) are  clear. Part (\ref{bd-3}) can be proven using results in \cite{Chebli2004}, but since our situation
is simpler we include a self contained proof here. We start by noting that $\cP_k$ is a 
finite dimensional subspace of $L^2(\mathbb{S}^{2n-1})$ and for each $z$ the map
$p\mapsto p(z)$ is linear and well defined and hence continuous. Thus $\cP_k$ is a reproducing kernel Hilbert space.
By Lemma 1.11 in \cite{Zhu2005} a polynomial of homogeneous degree $k$ satisfies
\begin{equation*}
  \| p \|^2_{\mathcal{V}_\sigma} 
  = \frac{\Gamma(n+k)\, \Gamma(\sigma)}{\Gamma(n)\,\Gamma(\sigma+k)}
  \int_{\mathbb{S}^{2n-1}} |p(z)|^2\,d\sigma_n(z) =  \frac{\Gamma(n+k)\, \Gamma(\sigma)}{\Gamma(n)\,\Gamma(\sigma+k)}
  \|p\|^2_{L^2(\S^{2n-1})}
\end{equation*} 
It follows from (\ref{eq:nonZgamma}) that
\[\|z^\gamma \|^2_{L^2(\mathbb{S}^{2n-1})}= \frac{\gamma !\, \Gamma (n)}{\Gamma (n+k  )}\, .\]
Thus  $\{(\Gamma (n+k)/\gamma! \Gamma (n))^{1/2}\, z^\gamma\}_{|\gamma|=k}$ is a orthonormal basis for
$\cP_k$ and the reproducing kernel $H_k$ is given by
\begin{equation}\label{eq:Hk}
H_k(z,w)=\sum_{|\gamma|=k} \frac{\Gamma (n +k)}{\gamma! \Gamma (n)}\, z^\gamma\overline{w}^\gamma\, ,\quad
\|H_k\|^2_{L^2(\mathbb{S}^{2n-1}\times \mathbb{S}^{2n-1})}=\dim \cP_k\le (1+k)^n\,  .
\end{equation}
As $H_k(k\cdot z,  w)= H_k(z,k^{-1}w)$ for all $k\in K$ and $z,w\in\nball$ it follows from the $K$-invariance of the measure $\sigma_n$ 
and (\ref{eq:Hk}) that
\[\int_{\mathbb{S}^{2n-1}}|H_k(z,w)|^2\, d\sigma_n(w)=\int_{\mathbb{S}^{2n-1}\times \mathbb{S}^{2n-1}}|H_k(z,w)|^2\, d\sigma_n\otimes \sigma_n(z,w) \le (1+k)^n\, .\]
Combining all of this, we get for $p\in\cP_k$ and $z\in \nball$:
\[
|p (z)| = \left|\int_{\mathbb{S}^{2n-1}}p(w)H_k(z,w)\, d\sigma_n(w)\right|\le (1+k)^{n/2}\, \|p\|_{L^2(\mathbb{S}^{2n-1})}\, .\]

For $\sigma > n$ we have $\Gamma(n+k)/\Gamma(\sigma+k) \sim k^n/k^\sigma  $
as $k\to \infty$. Hence, there exists constants   $C_1,C>0$ such that for all $k$
$$
\frac{\Gamma(n+k)}{\Gamma(\sigma+k)} 
\geq C_1 k^{n-\sigma}
\geq C(1+k)^{n-\sigma}\, .
$$ 
From this we see that
\[
  \int_{\mathbb{S}^{2n-1}} |p(z)|^2\,d\sigma_n(z) 
  \leq 
  C(1+k)^{\sigma-n}  \| p \|^2_{\mathcal{V}_\sigma} \, .
\]
Thus, if $\varphi =\sum_k \varphi_k$, with $\varphi_k$ homogeneous of degree $k$, it follows from Lemma \ref{eq:VsI} that
$$
|\varphi(z)|
\leq C \sum_k (1+k)^{\sigma/2} \| \varphi_k\|_{\mathcal{V}_\sigma}
\leq C_N \sum_k (1+k)^{\sigma/2} (1+k)^{-N}
$$
which is finite if $N$ is large enough.

  \end{proof}

\section{Bergman spaces, wavelets, frames 
  and atomic decompositions}

\noindent
We now present the general (coorbit) theory which 
allows us to use the discrete series representations to
provide atomic decompositions and frames for the Bergman spaces.

\subsection{Wavelets and coorbit theory}
\label{sec:coorbit}
In this section we describe the elements of coorbit theory with the
goal of applying them to the  representation $\pi _\sigma $ of $S$ or
$G_\sigma $. As $\pi _\sigma $ is highly reducible on $S$, we will use
the general set-up of
\cite{Christensen2009,Christensen2011,Christensen2012} rather than the
original work  in \cite{Feichtinger1989a}.

Let $\fS$ be a Fr\'echet space and let $\fS^\cdual$ be the 
conjugate linear dual equipped with the weak* topology, and
assume that $\fS$ is continuously embedded and weakly dense in $\fS^\cdual$. 
The conjugate dual pairing of elements $\psi\in \fS$ and $f\in \fS^\cdual$ will be denoted
by $\dup{f}{\psi}$. Let $G$ be a locally compact 
group with a fixed left Haar measure
$dx$, and assume that $(\pi,\fS)$ is a continuous representation of
$G$, i.e. $x\mapsto \pi(x)\psi$
is continuous for all $\psi\in \fS$. 
A vector $\psi\in \fS$ is called \emph{weakly cyclic} if 
$\dup{f}{\pi(x)\varphi}=0$ for all $x\in G$ implies that
$f=0$ in $\fS^*$.
The contragradient representation $(\pi^\cdual,\fS^\cdual)$ is 
the continuous representation of $G$ on $\fS^\cdual$ 
defined by
\begin{equation*}
  \dup{\pi^\cdual(g)f}{\varphi}=\dup{f}{\pi(g^{-1})\varphi} \text{\ for $f\in \fS^\cdual.$}
\end{equation*}
For a fixed vector 
$\psi\in \fS$ define the \emph{wavelet transform} $W_\psi:\fS^*\to C(G)$ by
\begin{equation*}
  W_\psi(f)(x) = \dup{f}{\pi(x)\psi} = \dup{\pi^*(x^{-1})f}{\psi}.
\end{equation*}
A cyclic vector $\psi$ 
is called an \emph{analyzing vector} for $\fS$ if for all $f\in \fS^*$ 
the following convolution reproducing formula holds
\begin{equation} \label{rf}
  W_\psi(f)*W_\psi(\psi) = W_\psi(f).
\end{equation}

The main achievement of coorbit theory is the discretization of the
reproducing formula~\eqref{rf} and  the construction of atomic
decompositions and Banach frames for the coorbit spaces~\cite{Feichtinger1988,Feichtinger1989a,Grochenig1991}. We formulate
these results in the set-up of~\cite{Christensen2012}, as it also
covers the case of non-integrable representations. 

The second ingredient of coorbit theory is a class of function spaces
on $G$. In the sequel $Y$ is a Banach function space on $G$ satisfying
the following properties~\cite{Christensen2011}: \\

(R1) $Y$ is a solid Banach function space continuously embedded
in $L^1_{loc}$. $Y$ is invariant under
left and right translation of $G$:
If  $  \ell_x
\widetilde{f}(y) = \widetilde{f}(x^{-1}y)$ denotes the left
translation of a  function  $\widetilde{f}$ on $G$ and $x,y \in G$, 
$\| \ell_x \widetilde{f}\|_Y\leq C_x \| \widetilde{f}\|_Y$ for all
$f\in Y$ for some weight function $C_x$ on $G$. 
Moreover, for each $\widetilde{f}\in Y$, the mapping $x\mapsto \ell_x \widetilde{f}$ is
continuous $G\to Y$. Likewise for the right translation.

(R2) For a fixed analyzing vector $\psi \in \mathcal{S}$ the sequilinear mapping
$(\tilde f, \varphi ) \mapsto \int_G \widetilde{f}(x) \dup{\pi^*(x)\psi}{\varphi}\,
dx \in\mathbb{C}$ is continuous on $Y \times \mathcal{S}$. 

(R3) For a fixed analyzing vector $\psi$, the convolution operators 
$\tilde f \to \tilde f \ast |W_\varphi (\psi)|$ and 
$\tilde f \to \tilde f \ast |W_\psi (\varphi)|$
are bounded on $Y$ for all $\varphi \in \mathcal{S}$. 

In our case, $Y$ will be a weighted $L^p$-space, and the invariance
properties (R1) are obviously satisfied.

For an analyzing vector $\psi$ define the subspace $Y_\psi$ of $Y$ by 
$$Y_\psi = \{ \widetilde{f}\in Y\, |\, \widetilde{f}=\widetilde{f}*W_\psi(\psi) \},$$ 
and define the \emph{coorbit space} with respect to the wavelet $\psi$
to be 
$$\Co_{\fS}^\psi Y = \{ f\in \fS^*\, |\, W_\psi(f)\in Y  \}$$
equipped with the norm $\| f \| = \| W_\psi(f)\|_Y$.

The following result from \cite{Christensen2011} lists conditions which
ensure that $\Co_{\fS}^\psi Y$ is a Banach space that  is isometrically
isomorphic to $Y_\psi$. Even weaker  conditions can be found in
\cite{Christensen2011}. 
\begin{theorem}
  \label{thm:coorbitsduality} 
  
  Assume $\psi$ is
  an analyzing vector for $\fS$ and $Y$ satisfies the axioms (R1) ---
  (R3). 
  Then 
  \begin{enumerate}
  \item $Y_\psi$ is
    a closed reproducing kernel subspace of $Y$ with reproducing kernel
    $K(x,y)=W_\psi(\psi)(x^{-1}y)$.\label{prop1}
  \item The space $\Co_{\fS}^\psi Y$
    is a $\pi^\cdual$-invariant Banach space.     \label{prop2}
  \item $W_\psi:\Co_{\fS}^\psi Y\to Y$ intertwines $\pi^\cdual$ and
    the left
    translation. \label{prop3}
  \item If left translation is continuous on $Y,$ then
    $\pi^*$ acts continuously on $\Co_{\fS}^\psi Y.$\label{prop4}
  \item $\Co_{\fS}^\psi Y = \{ \pi^\cdual(\widetilde{f})\psi \mid \widetilde{f}\in
    Y_\psi\}$. \label{prop5}
  \item $W_\psi:\Co_{\fS}^\psi Y \to Y_\psi$ is an isometric
    isomorphism.\label{prop6}
  \end{enumerate}
\end{theorem}
 
To summarize, in order to apply the abstract theory to a concrete
representation $\pi $ of a locally compact group, we need to verify that the
representation possesses an analyzing vector $\psi $ in a suitable space of
test functions $S$, such the the reproducing formula \eqref{rf}
holds, and that the function spaces satisfies the axioms (R1) ---
(R3). The reproducing formula  is a non-trivial property and is usually related to the 
square-integrability of $\pi $, see~\cite{fuehr}. For non-integrable
representions the boundedness of the convolution operators in (R2) may 
also be rather subtle (see Proposition~\ref{prop:38}).

In our case the starting point is the
representation $(\pi_\sigma,\cV_\sigma)$ of $S$ and of $G_\sigma $. We take $\cS=\cV_\sigma^\infty$ and note that
$\cS$ is contained in the space of smooth vectors for $\left. \pi_\sigma \right|_{S}$ and is invariant
under the group $S$ and its Lie algebra. 

To formulate sufficient density conditions  for discrete sets in groups and general
sequence spaces, we recall that 
a set $X=\{x_i : i\in I\}\subseteq G$ is called relatively separated in $G$, if
$\max _{g\in G} \, \mathrm{card} \, (X \cap gU)  < \infty $ for some
(all) compact neighborhood  $U\subseteq G$. Given a neighborhood $U\subseteq G$,  $X$ is called $U$-dense, if $\bigcup
_{i\in I} x_i U  = G$.

The following theorem is a general version of the main theorem of
coorbit theory and yields the existence of atomic decompositions and
frames for coorbit spaces. 

\begin{theorem} \label{at}
Let $G$ be a connected Lie group and $(\pi ,\cH)$ a unitary representation of $G$. 
Assume that $\cS$ is a Fr\'echet space continuously embedded 
in $\cH$,
and that $\cS$ is invariant under $G$ and its Lie algebra.
Let $\psi \in \mathcal{S}$ be an  analyzing vector and   $Y$ be a Banach space of functions on $G$ satisfying the
conditions (R1) --- (R3).  \\

(A) \emph{Atomic   decompositions for coorbit spaces.}
There exists a neighborhood $U\subseteq G$, such that every
relatively separated and $U$-dense set $X= \{x_i\mid i\in I\}$ generates
an atomic decomposition  $\{\pi (x_i)\psi \mid i\in I\}$ for $\Co _{\mathcal{S}} ^\psi
Y$. This means that there exist constants $A,B >0$ and  a set of
functionals $\{c_i \mid i\in I\} \subseteq  (\Co _\pi ^\psi
Y)^*$, such that every $f\in \Co _{\mathcal{S}}^\psi
Y$ possesses the expansion
\begin{equation}
  \label{eq:c5}
f = \sum _{i\in I} c_i(f) \pi (x_i)\psi \, .  
\end{equation}
The series  \eqref{eq:c5} converges unconditionally  in $\Co _{\mathcal{S}} ^\psi Y$, whenever the
functions with compact support are dense in $Y$, otherwise the series
is 
weak$^*$-convergent in $S^*$.  The coefficient sequence satisfies
$$
 \|\sum _{i\in I} c_i(f)  1_{x_iU} \|_Y \leq A \|f\|_{\Co _{\mathcal{S}}^\psi
   Y} \, . 
$$
Conversely, if $ \sum _{i\in I} a_i  1_{x_iU} \in Y$, then
$
h=  \sum _{i\in I} a_i  \pi
(x_i)\psi
$ is in $\Co _{\mathcal{S}}^\psi Y$ and $\|h\|_{\Co _{\mathcal{S}}^\psi Y} \leq  B
\|\sum _{i\in I} a_i  1_{x_iU} \|_Y$. \\

(B) \emph{Banach frames for coorbit spaces.} 
There exists a neighborhood $U\subseteq G$, such that every
  relatively separated and $U$-dense set $X= \{x_i\mid i\in I\}$ generates
  a Banach frame $\{\pi (x_i)\psi \mid i\in I\}$ for $\Co _{\mathcal{S}} ^\psi
  Y$. This means that there exist constants $A,B >0$ and  a dual frame
  $\{ e_i \mid i \in I\}$, such that, for all $f\in \Co _{\mathcal{S}} ^\psi Y$
  $$
  A \|f\|_{\Co _{\mathcal{S}} ^\psi Y} \leq \|\sum _{i\in I} \langle f, \pi
  (x_i)\psi \rangle 1_{x_iU} \|_Y \leq B \|f\|_{\Co _{\mathcal{S}} ^\psi Y} \, ,
  $$
  and 
  $$
  f= \sum _{i\in I} \langle f, \pi
  (x_i)\psi \rangle e_i
  $$
  with unconditional convergence in $\Co _{\mathcal{S}} ^\psi Y$, when the
  functions with compact support are dense in $Y$ and weak$^*$ in $S^*$ 
  otherwise.
\end{theorem}

\begin{proof}
(A) Under  the stated conditions the reproducing kernel space
  $Y_\psi$ with reproducing kernel $\Phi= W_\psi(\psi)$ satisfies
  both Theorem 2.13 and 2.14 from \cite{Christensen2012}. Note that by unitarity
  left and right derivatives of the kernel (refer to \cite{Christensen2012}) 
  correspond respectively to the functions
  $W_\psi(\pi^\infty(X)\psi)$ and $W_{\pi^\infty(X)\psi}(\psi)(x)$ 
  for $X\in\mathfrak{g}$. Since $\mathcal{S}$ is $\pi^\infty$-invariant
  we therefore see that convolution with left or right derivatives
  of $\Phi$ will be continuous on $Y$. This is enough to ensure the
  invertibility of the appropriate operators in \cite{Christensen2012} leading to
  atomic decompositions.

(B)   Since right differentiation of the kernel $\Phi = W_\psi(\psi)$
  corresponds to the function $W_{\pi^\infty(X)\psi}(\psi)$
  and $\mathcal{S}$ is $\pi^\infty$-invariant, the  conditions
  ensure that the required convolution in Lemma 2.8 from \cite{Christensen2012}
  is continuous. Therefore Theorem 2.11 from \cite{Christensen2012} provides
  the desired frames.
\end{proof}

\begin{remark} If $Y$ is a weighted Lebesgue space, 
$Y = L^p_v(G)$, then $\|\sum _{i\in I} c_i 1_{x_iU} \|_{L^p_v}$ is equivalent to the weighted $\ell
^p$-norm
$\|c\|_{\ell ^p_v} =$ $  \Big(\sum _{i\in I} |c_i|^p
 v(x_i)^p \Big)^{1/p}$. Thus the implicit definition of the sequence
 space norm is perfectly natural. 
\end{remark}

\subsection{Coorbits for the discrete series representation}
\label{sec:coorbit-description}

In the following sections we apply the abstract theory to the
representation $\pi _\sigma $ on either $G_\sigma $ or $S$ and will
show that the weighted Bergman spaces arise as  coorbits 
for the groups $S$ and $G_\sigma$. Then we derive 
atomic decompositions and frames for these spaces. We will mainly use
the set-up in  \cite{Christensen2012}. As mentioned earlier we take $\cS$ for the
space of smooth vectors in $\cV_\sigma$. For the moment we also fix the notation
$\psi=1_{\nball}\in \mathcal{V}_\sigma^\infty$.

We start by a simple observation. Let $x\in G_\sigma$ and set $w=x\cdot 0$.
Let $f  \in \mathcal{V}_\sigma^{-\infty}$ and write $f=\sum_k f_k$ as in (\ref{eq:VsI1}) and note that this
sum converges in the weak topolgy.  Hence we may  interchange the
sum with $f$ and  obtain
\begin{equation}
  \label{waveletcoeff1}
  W^\sigma_\psi(f)(x)
  = \sum_k W_{\psi}^\sigma (f_k)(x)
  = \sum_k \frac{1}{d^\sigma} f_k(w) 
  = \frac{1}{d^\sigma} f(w),
\end{equation}
since each $f_k$ is in $\mathcal{V}_\sigma$.

\begin{proposition}

  \label{prop:integrability}
  Let $\psi=1_{\nball}$ and let $\sigma>n$, then $\psi$ is both $S$-cyclic 
  and $G_\sigma$-cyclic in $\mathcal{V}_\sigma^\infty$. 
  If $\varphi $ is a 
  smooth vector, then there is a constant
  $C$ such that
  $|W_\psi^\sigma(\varphi)| \leq C |W_\psi^\sigma(\psi)|$. Furthermore, for
  $1\leq p<\infty$ the function
  $W_\psi^\sigma(\varphi)$ is in $L^p_t(S)$ for $t+p\sigma /2 > n$.
  If $\sigma$ is rational,  then also  
  $W_\psi^\sigma(\varphi) \in L^p_t(G_\sigma)$ for the same range of parameters.
\end{proposition}

\begin{proof}
  Cyclicity for both $S$ and $G_\sigma$ follows from equation (\ref{waveletcoeff1}).
  By Lemma \ref{le:1andK} and the fact that  $\varphi\in \mathcal{V}_\sigma$ we obtain that  
  $W_\psi^\sigma(\varphi)(x) 
  = \frac{1}{d^\sigma} \varphi(x\cdot o)
  = W_\psi^\sigma(\psi)(x) \varphi(x\cdot o)$.
Since $\varphi$ is bounded on $\nball$ by Lemma~\ref{bd}, the estimate
  follows. 
By Corollary \ref{rkrep} we have 
  \begin{equation*}
    |W_\psi^\sigma(\psi)(x)| 
    = (1-|x\cdot o|^2)^{\sigma/2}\, ,
  \end{equation*}
  so $|W_\psi^\sigma (\psi)|$ is in $L^p_t(S)$ if and only if  the following integral 
  is finite
  \begin{align}
    \int_{S} (1-|x\cdot o|^2)^{p\sigma/2+t} \,dx
    = \int_{\nball} (1-|z|^2)^{t+p \sigma/2-n-1}\,dv(z),
  \end{align}
  which happens when $t+p\sigma/2 > n$. If $\sigma>n$ is rational,
  then  the integral over $G_\sigma$ reduces to a finite multiple
  of the integral over $S$, and  the proof is complete. 
  \end{proof}

\subsection{Bergman spaces as coorbits for the group $S$}

In this section we verify the assumptions of Theorem~\ref{thm:coorbitsduality} for the
representations $\pi _\sigma $ on $S$ and then  identify the weighted Bergman spaces as coorbit
spaces with respect to  $\pi _\sigma |_S $.  For simplicity we write $\pi_\sigma$ for $\pi_\sigma |_S$ as
it will always be clear if we are using $S$ or $G_\sigma$. 

First we show that the coorbit spaces $\Co L^p_{\alpha+n+1-\sigma
  p/2}(S)$ with respect to the representation $\pi _\sigma $ are
well defined. In this case we deal with a weighted $L^p$ space on $S$
with weight $(1- |x\cdot o|^2)^t $ for some $t$. This space is
invariant under left and right translations because the weight is
easily seen to be moderate (a property which is equivalent to
translation invariance).

\begin{proposition}
  \label{prop:reproducing}
  Let $\psi=1_{\nball}$, $\sigma>n$ and $1\leq p <\infty$.

(i)  Then there is a positive 
  constant $C$ such that 
    $$
  W_\psi^\sigma(f)*W_\psi^\sigma(\psi) = CW_\psi^\sigma(f) \qquad
  \text{ for every } f\in \mathcal{V}_\sigma^{-\infty} \, ,
  $$
  where convolution is on the group $S$.  Thus
  $\sqrt{C}\psi$ is an analyzing vector for $\pi _\sigma $.

(ii) If $-1 <\alpha< p(\sigma-n)-1$, then  the mapping 
  $$
  L^p_{\alpha+n+1-\sigma p/2}(S)\times \mathcal{V}_\sigma^\infty \ni (\wf,\varphi)
  \mapsto \int_S \wf(x)\dup{\pi^*(x)\psi}{\varphi}\,dx \in \mathbb{C}
  $$
  is continuous. 

As a consequence,  the coorbit space  
  $\Co_{\mathcal{V}_\sigma^{\infty}}^\psi L^p_{\alpha+n+1-\sigma
    p/2}(S)$ is a non-trivial  Banach space
  for $-1<\alpha < p(\sigma-n)-1$.
\end{proposition}
\begin{proof}
  Cyclicity of $\psi$ has been stated in Proposition~\ref{resttos}.
 Next, we verify the reproducing formula. Let   $x=\begin{pmatrix}
    a_x & b_x \\ c_x^t & d_x
  \end{pmatrix} \in SU(n,1)$ and   $w=x\cdot o$, and likewise
  $y=\begin{pmatrix}
    a_y & b_y \\ c_y^t & d_y
  \end{pmatrix}$ with $z=y\cdot o$.   Then the proof of Corollary \ref{rkrep} gives
  $1-|w|^2=1/|d_x|^2$.  A similar computation yields that 
 $d_{y^{-1}x}=\overline{d}_yd_x (1-\dup{w}{z})$.
If $f\in \mathcal{V}_\sigma^{-\infty}$, then   by
(\ref{waveletcoeff1}) and Lemma~\ref{conv} we obtain that 
  \begin{align}
    W_\psi^\sigma(f)*W_\psi^\sigma(\psi)(x)
    &= \int_{S} \frac{1}{d_y^\sigma} f(y\cdot o) 
    \frac{1}{(d_{y^{-1}x})^\sigma} \,dy \notag \\
    &=C \frac{1}{d_x^\sigma} \int_{\nball}   
    f(z)\frac{(1-|z|^2)^\sigma}{(1-\dup{w}{z})^\sigma} 
    \frac{dv(z)}{(1-|z|^2)^{n+1}} \notag \\
    &=CW_\psi^\sigma (f)(x). \label{eq:c6}
    \end{align}

    For the continuity it is enough to show that
    $\wf\mapsto \int |\wf(x)W_\psi^\sigma(\psi)(x)|\,dx$ is continuous.
    This will follow  provided  that $W_\psi^\sigma(\psi)$ is in
    the K\"othe dual of $L^p_{\alpha+n+1-\sigma p/2}(S)$
    which is $L^q_{\sigma q/2-(\alpha+n+1)q/p}(S)$ where $1/p+1/q=1$. 
    According to Proposition~\ref{prop:integrability}
    this is the case  if
    $\sigma q/2-(\alpha+n+1)q/p +\sigma q/2 > n$ which is
    equivalent to $\alpha < p(\sigma-n)-1$.
    
  Finally,  the coorbit space is non-trivial, because
    $W_\psi^\sigma(\psi) \in L^p_{\alpha+n+1-\sigma p/2}(S)$ by Proposition~\ref{prop:integrability}, and therefore
    $\psi \in \Co _{\pi _\sigma} ^\psi L^p_{\alpha+n+1-\sigma p/2}(S)$.
  \end{proof}

Next we identify the classical Bergman spaces on $\nball $ with the
coorbit spaces of the representation $\pi _\sigma$. 

\begin{theorem} \label{identberg}
  Let $\alpha >-1$, $1\leq p < \infty $,  and  $\psi=1_{\nball}$.
  A holomorphic function $f$ 
  is in $\mathcal{A}^p_\alpha$ if and only if 
  $f\in \mathcal{V}_\sigma^{-\infty}$ and $W_\psi^\sigma(f)$ is in $L^p_{\alpha+n+1-\sigma p/2}(S)$.
  Therefore the Bergman space $\mathcal{A}^p_\alpha$ is equal to the coorbit space
  $\Co_{\mathcal{V}_\sigma^{\infty}}^\psi  L^p_{\alpha+n+1-\sigma p/2}(S)$ for $-1<\alpha < p(\sigma-n)-1$.
\end{theorem}
\begin{proof}
First, $\Co_{\mathcal{V}_\sigma^{\infty}}^\psi  L^p_{\alpha+n+1-\sigma
  p/2}(S)$ is well defined by Proposition~\ref{prop:reproducing}. 
  Let $x\in S$ and  $w=x\cdot o$. Then  
   $1-|x\cdot o|^2 = 1-|w|^2= 1/|d|^2$. If $f\in
  \mathcal{V}_\sigma^{-\infty}$, then we have by Lemma \ref{le:1andK}
  and \eqref{waveletcoeff1} 
  \[|f(w)| = (1-|x\cdot o|^2)^{- \sigma/2} |W_\psi(f)(x)|\, .\]
  Using 
  the isometry  (\ref{eq:normequiv}) we see that $f\in L^p_\alpha
  (\nball )$, if and only if $W_\psi (f)  \in L^p_{\alpha+n+1-\sigma p/2}(S)$.

It remains to be shown that $\mathcal{A}^p_\alpha \subseteq V_\sigma
^{-\infty}$. Since the  measure $dv_\alpha$ is a probability measure, we have
$\mathcal{A}^p_\alpha \subseteq \mathcal{A}^1_\alpha$ for all $p\geq 1$.
  So it suffices to show that   $ {\cA}^1_\alpha \subseteq \mathcal{V}_\sigma^{-\infty}$
  when $\alpha>-1$ and $\sigma>n$.  
We proceed as follows:   there is a $\beta>\alpha$ such that
  $\mathcal{A}^1_{\alpha } \subseteq \mathcal{A}^2_\beta $ (precisely  $\beta = n+1+2\alpha$ by  exercise 2.27 on p. 78 in
  \cite{Zhu2005}),   so we now assume $f\in \mathcal{A}^2_\beta$.
  The Taylor  expansion for $f$, in multi-index notation,  is
  \begin{equation*}
    f(z) =  \sum_{k=0}^\infty 
    \sum_{|\gamma |=k} \frac{1}{\gamma !} \frac{\partial^\gamma  f}{\partial z^\gamma }(0) z^\gamma
  =\sum_{k=0}^\infty f_k(z)  \, .
  \end{equation*} 
  
  The
  Cauchy-Schwartz  inequality implies that
    \begin{equation*}
    |f_k(z)|^2 \leq \dim (\mathcal{P}_k) \,\, 
    \sum_{|\gamma |=k}  \left| \frac{\partial^\gamma  f}{\partial
        z^\gamma }(0)\right|^2 \left| \frac{z^\gamma }{\gamma
        !}\right|^2.  
  \end{equation*}
 
 Since $\|z^\gamma \|_{\mathcal{V}_\sigma}^2= \|z^\gamma \|_{\sigma -n
   -1}^2 =  \tfrac{\gamma !\,  \Gamma
  (\sigma) }{\Gamma(\sigma +|\gamma|)}$ by (\ref{eq:nonZgamma}), 
we obtain that 
  \begin{align*}
    \| f_k\|_{\mathcal{V}_\sigma}^2 
    &\leq \dim (\mathcal{P}_k) \,\,  \sum_{|\gamma |=k} \left| \frac{\partial^\gamma f}{\partial z^\gamma }(0)\right|^2 
    \|z^\gamma\|_{\mathcal{V}_\sigma } ^2 \\
    &= \dim (\mathcal{P}_k) \,\, \frac{\Gamma(\sigma )}{\Gamma(\sigma+k)} 
    \sum_{|\gamma |=k} \left| \frac{\partial^\gamma f}{\partial z^\gamma }(0)\right|^2 
    \frac{1}{\gamma !}.
  \end{align*}
  Next we differentiate the reproducing formula for
  $\mathcal{A}^2_\beta $ in Lemma \ref{le:repK}
  and note that we can interchange the integration and differentiation
  to get 
    \begin{align*}
    \frac{\partial^\gamma f}{\partial z^\gamma}(0) 
  &  = \frac{\Gamma(\beta+n+1+k)}{\Gamma(\beta+n+1)} \int_{\nball}
  f(w)\overline{w}^\gamma \,dv_\beta(w) \, . 
  \end{align*}
Using the orthonormal basis $\varphi _\gamma (z) =
\Big(\frac{\Gamma(\beta+n+1+k)}{\gamma ! \Gamma(\beta+n+1)}\Big)^{1/2}
z^\gamma $ for $\mathcal{A}^2_\beta $, we see that 
$$
\frac{\Gamma(\beta+n+1+k)}{\gamma !\, \Gamma(\beta+n+1)} \Big|
\int_{\nball} f(w)\overline{w}^\gamma \,dv_\beta(w) \Big| ^2 = | \langle
f, \varphi _\gamma \rangle _{\mathcal{A}^2_\beta } |^2
$$ and  consequently 
    \begin{align*}
    \| f_k\|_{\mathcal{V}_\sigma}^2  &\leq 
    \dim (\mathcal{P}_k) \frac{\Gamma(\beta+n+1+k)}{\Gamma(\beta+n+1)} 
    \frac{\Gamma(\sigma)}{\Gamma(\sigma+k)} 
    \sum_{|\gamma |=k} | \langle f, \varphi _\gamma \rangle
    _{\mathcal{A}^2_\beta } |^2  \\
&\leq  \dim (\mathcal{P}_k) \frac{\Gamma(\beta+n+1+k)}{\Gamma(\beta+n+1)} 
    \frac{\Gamma(\sigma)}{\Gamma(\sigma+k)}  \|f\|   _{\mathcal{A}^2_\beta } ^2
  \end{align*}

  Since  clearly the sequence 
  \begin{align*}
    c_k
    &= \frac{\Gamma(\beta+n+1+k)}{\Gamma(\beta+n+1)}
    \frac{\Gamma(\sigma)}{\Gamma(\sigma+k)}
  \end{align*}
  is bounded in $k$  and $\dim (\mathcal{P}_k) \le (1+k)^n$, it
  follows that 
$$
\|f_k \|_{\mathcal{V}_\sigma } ^2 \leq C (1+k)^n \, ,
$$ 
and therefore $f= \sum _k f_k \in \mathcal{V}_\sigma^{-\infty}$ by the
definition of $\mathcal{V}_\sigma^{-\infty}$.
\end{proof}

Note that the Bergman space $\mathcal{A}^p_\alpha $ can be represented
as a coorbit  space with respect to \emph{all} representations $\pi
_\sigma $  for which $\sigma > n+ \tfrac{\alpha + 1}{p}$. 
This independence of the representation is an interesting and
surprising phenomenon and is perhaps related to the semisimplicity of
$G_\sigma$.

\subsection{Atomic decompositions of Bergman spaces  via the group $S$}
\label{sec:atomdecomp}

In this section we construct  frames and atomic decompositions
for the Bergman  spaces by applying the abstract Theorem~\ref{at}. 
 
 We first need  
to verify the continuity of
the convolution operator  $f \to f \ast |W^\sigma_\psi(\psi)|$. 
For that we recall  the following well known fact from \cite[Thm. 2.10]{Zhu2005}:
\begin{theorem} 
 \label{Thm:Zhu1}
  Fix two real parameters $a$ and $b$ and define the integral operator
  $\mathbf{S}$ by
  \begin{equation*}
    \mathbf{S}f(z) = (1-|z|^2)^a \int_{\nball} \frac{(1-|w|^2)^b}{|1-\dup{z}{w}|^{n+1+a+b}} f(w)\,dv(w).
  \end{equation*}  

Then, for $t$ real and $1\leq p<\infty$, $\mathbf{S} $ is bounded on
$L^p_t(\nball)$ if and only if  $-pa<t+1<p(b+1)$.
\end{theorem} 

\begin{proposition} \label{prop:38}
If $1\leq p < \infty $ and $-1<\alpha <p(\sigma-n)-1$, then the convolution operator
$\wf\mapsto \wf*|W_\psi^\sigma(\psi)|$ is continuous  
  on $L^p_{\alpha+n+1-\sigma p/2}(S)$.  
\end{proposition}
\begin{proof}
  Recall that if $w=x\cdot o$,
  then
  $1-|w|^2 = 1/|d_x|^2$.
  Also, if $z=y\cdot o$,
  then $d_{y^{-1}x}=\overline{d}_yd_x (1-\dup{w}{z})$. By
  Corollary~\ref{rkrep}
$|W_\psi (\psi )(x)| = |d_{x}|^{-\sigma }$. 
 For  $f 
  \in L^p_{\alpha-\sigma p/2}(\nball)$   let $\wf (x) = f(x\cdot o) $
  be the corresponding function in $ L^p_{\alpha+n+1-\sigma p/2}(S)$.
  By repeating the estimate in \eqref{eq:c6} we have 
  \begin{align*}
    (\wf*|W_\psi^\sigma(\psi)|)(x)
    &= \int_S \wf(y)|W^\sigma _\psi (\psi) (y^{-1}x)|\,dy \\
    &= \int_S \wf(y)\frac{1}{|d_{y^{-1}x}|^\sigma}\,dy \\
    &= (1-|w|^2)^{\sigma/2}\int_{\nball} f(z) 
    \frac{(1-|z|^2)^{\sigma/2}}{|1-\dup{w}{z}|^\sigma} 
    \frac{dv(z)}{(1-|z|^2)^{n+1}}.
  \end{align*}
  This convolution  operator corresponds to $\mathbf{S}$ in Theorem~\ref{Thm:Zhu1} with 
  $a=\sigma/2$, $b=\sigma/2-n-1$. Applying  Theorem~\ref{Thm:Zhu1}, we
  see that this operator is bounded on $L^p_{\alpha-\sigma
    p/2}(\nball)$ and thus   on $L^p_{\alpha+n+1-\sigma p/2}(S)$,  if and only if
  $-p\sigma /2 < \alpha -\sigma p/2 +1 < p(\sigma/2-n)$.
  Since  $\alpha>-1$, 
  the first inequality is automatically satisfied, whereas  the second
  inequality is equivalent to $\alpha < p(\sigma-n)-1.$
\end{proof}  

\begin{remark}
  The boundedness of the convolution operators above can also  be proven
  directly with  group theoretic methods. 
  For example, in the unweighted case $\alpha+n+1-\sigma p/2 =0$  the continuity
  follows as long as the kernel $W^\sigma_\psi(\psi)$ 
  is in the weighted space $L^1_{\Delta^{(p-1)/p}}(S)$, where $\Delta$
  is the modular function on $S$. This result, and the relation to 
  the Kunze-Stein phenomenon, will be discussed in the forthcoming article \cite{Christensen2016}.
  However, since Theorem~2.10 from \cite{Zhu2005} is already known, we 
  have decided to use it instead.
\end{remark}


The  restriction on $\sigma $ is the same as in the work of Coifman and
Rochberg~\cite{Coifman1980} and Zhu~\cite[Thm.~2.30]{Zhu2005}.  Our approach gives a group
theoretic reason for this  condition on $\sigma $.

{}From the  estimates of the wavelet coefficients 
in Proposition~\ref{prop:integrability}, we immediately obtain  the following statement.
\begin{corollary} \label{corc1}
 Let $1\leq p <\infty$. 
 If $\varphi $ is a  smooth vector  for $\pi_\sigma$ and  $-1<\alpha <p(\sigma-n)-1$, then the
 convolution operators $\wf\mapsto \wf*|W_\varphi^\sigma(\psi)|$ and
  $\wf\mapsto \wf*|W_\psi^\sigma(\varphi)|$ are continuous 
  on $L^p_{\alpha+n+1-\sigma p/2}(S)$.
\end{corollary}

To apply the discretization machinery of coorbit theory, we finally need a procedure
to construct suitable dense sets in the group $S$. In Lie groups this can
be done explicitly as follows: 
Let $X_1,\dots,X_n$ be  a basis for the Lie algebra of $S$ and define
a neighborhood 
$U_\epsilon$ of the identity matrix  by
\begin{equation*}
      U_\epsilon 
      = \{ e^{t_1X_1}\cdots e^{t_nX_n} 
      \mid -\epsilon \leq t_k\leq \epsilon    \} \, .
\end{equation*}
Then a set $\{x_i: i\in I\}$ is $U_\epsilon $-dense, if $\bigcup
_{i\in I } x_i U_\epsilon = G$. 


Applying Theorem~\ref{at} (B)  to the representation $\pi _\sigma $ of
$S$, we obtain a first version of Banach frames for the Bergman
spaces. 
\begin{theorem} \label{bfrberg}
  Let $1\leq p< \infty $ and assume that  $-1<\alpha < p(\sigma-n)-1.$ Fix $\psi = 1_{\nball }$. Then  there exists
  a   set $X=\{x_i: i\in I\} \subseteq S$ with corresponding points
  $w_i = x_i\cdot o\in \nball $, such that $\{\pi _\sigma (x_i) \psi :
  i\in I\}$ is a Banach frame for $\mathcal{A}_\alpha ^p$. This means
  that there exist constants $A,B>0$ and a dual system $\{\widetilde \psi
  _i: i\in I\}$, such that 
$$
A \|f\|_{\mathcal{A}^p_\alpha } \leq \sum _{i\in I} |\langle f, \pi
_\sigma (x_i)\psi \rangle |^p (1-|x_i \cdot o|^2)^{\alpha +n+1-\sigma p/2}  \leq B
\|f\|_{\mathcal{A}^p_\alpha } \qquad \text{ for all } f\in
\mathcal{A}^p_\alpha  \, .
$$
Every $f \in \mathcal{A}^p_\alpha $ can be   reconstructed  by series 
$$
f= \sum _{i\in I} \langle f, \pi
_\sigma (x_i)\psi \rangle \widetilde \psi _i 
$$ 
with unconditional convergence in $\mathcal{A}^p_\alpha $. 
\end{theorem}


\begin{remark}
  Notice that for $p=\infty$ Theorem~\ref{at} cannot be used, since
  neither right nor left translation are continuous on the space $L^\infty_\alpha$.
  Instead one can use Remark 2.15 from \cite{Christensen2011} or the original coorbit theory
  which requires integrable kernels~\cite{Feichtinger1989a}. 
\end{remark}

The proof shows a bit more than just the existence of a sampling set  $X$. In fact,
for  $\epsilon >0$ small enough 
\emph{every} relatively separated and $U_\epsilon $-dense set $X\subseteq S$
generates a Banach frame $\{\pi _\sigma (x_i)\psi : i \in I\}$. 
Since by Corollary~\ref{rkrep} $\langle f, \pi _\sigma (x) \psi \rangle =
(1-|w|^2)^{\sigma /2} f(w)$, the existence of Banach frames for
$\mathcal{A}^p_\alpha $ is in fact a sampling theorem in disguise. 

\begin{corollary}
Assume that  $-1<\alpha < p(\sigma-n)-1$ and
  $1\leq p < \infty $.  Then  there exists
  a   set $W=\{w_i: i\in I\} \subseteq \nball$ and constants $A,B >0$,
  such that  
$$
A \|f\|_{\mathcal{A}^p_\alpha } \leq \sum _{i\in I} | f(w_i)|^p
(1-|w_i|^2)^{\alpha +n+1}   \leq B
\|f\|_{\mathcal{A}^p_\alpha } \qquad \text{ for all } f\in
\mathcal{A}^p_\alpha  \, .
$$
For every $f \in \mathcal{A}^p_\alpha $ we have the   reconstruction formula 
$$
f= \sum _{i\in I} f(w_i) (1-|w_i|^2)^{\sigma /2}  \widetilde \psi _i 
$$ 
with unconditional convergence in $\mathcal{A}^p_\alpha $. 
\end{corollary}

The dual version of the sampling theorem leads to the well known
atomic decomposition of Bergman spaces by Coifman and
Rochberg~\cite{Coifman1980,Luecking1985}.

\begin{theorem}[Atomic decomposition] \label{atberg}
Assume that  $-1<\alpha < p(\sigma-n)-1$ and
  $1\leq p< \infty $.  Then  there exists
  a   set $W=\{w_i: i\in I\} \subseteq \nball$ and a set of
  coefficient functionals $\{ c_i :i\in I\} \subseteq
  (\mathcal{A}^p_\alpha )^*$,  
  such that  every $f\in \mathcal{A}^p_\alpha $ possesses the
  decomposition 
$$
f(z)= \sum _{i\in I} c_i(f) (1-|w_i|^2)^{\sigma /2} (1-\langle z,
w_i\rangle )^{-\sigma }
$$ 
with unconditional convergence in $\mathcal{A}^p_\alpha $. 
The coefficient sequence $c$ is in $\ell ^p _{\alpha+n+1-\sigma
  p/2}(W)$ and satisfies the norm equivalence
\begin{equation}
  \label{eq:c2}
A \|f\|_{\mathcal{A}^p_\alpha } \leq \sum _{i\in I} | c_i(f)|^p
(1-|w_i|^2)^{\alpha +n+1- \sigma p/2 }   \leq B
\|f\|_{\mathcal{A}^p_\alpha } \qquad \text{ for all } f\in
\mathcal{A}^p_\alpha  \, .
\end{equation}
\end{theorem}

\begin{proof}
This is just    Theorem~\ref{at}(A) for the concrete representation $\pi
_\sigma $ on $S$ and the identification of $\mathcal{A}^p_\alpha $
with the  coorbit space $ \Co _{\pi
    _\sigma } ^\psi L^p_{\alpha +n+1- \sigma p/2}(S)$. Consequently
  every $f\in \mathcal{A}^p_\alpha 
  $ possesses an atomic decomposition of the form
$$
f= \sum _{i\in I} c_i(f) \pi _\sigma (x_i) \psi 
$$ 
with a coefficient sequence depending linearly on $f$ and satisfying
\eqref{eq:c2}. Since $\pi _\sigma (x_i) \psi = (1-|w_i|^2)^{\sigma /2} (1-\langle z,
w_i\rangle )^{-\sigma }$, this is the theorem, as stated. 
\end{proof}

Again, coorbit theory asserts that \emph{every } $U_\epsilon $-dense and
separated subset of $S$ (for $\epsilon $ small enough) generates an
atomic decomposition (with $w_i = x_i \cdot o$). 

We would also like to remark that in principle the coefficients $c_i(f)$ and the dual
Banach frame can be  calculated by inverting a frame-like
operator. See~\cite{Feichtinger1988,Grochenig1991} for details.


\subsection{Bergman spaces as coorbits 
  for the group $G_\sigma$ when $\sigma$ is rational}
So far we have worked with the group $S$ instead of the full group
$G_\sigma $.  The  advantage of using $S$  is that  there is no restriction on the
parameter $\sigma$ (except square-integrability forcing $\sigma >n$). 
On the other hand, the representation $\pi _\sigma | _S$ is
reducible on $\mathcal{A}^2_{\sigma -n-1}$, and thus the derivation of
the  reproducing formula is non-trivial and was done only for the
special analytic vector $\psi = 1_{\nball }$. 

In this section we study the Bergman spaces as coorbit spaces with
respect to the full group $G_\sigma $. In this case, we have to impose
the restrictions $\sigma > n$ \emph{and } $\sigma $ being rational so
that  $\pi _\sigma $ is 
square-integrable. Since $\pi _\sigma $  is irreducible and square integrable, the
orthogonality relations (\cite[Thm. 14.3.3]{Dix} or
\cite{fuehr,Feichtinger1989a}) 
 imply that the 
reproducing formula \eqref{rf} holds automatically for all (appropriately normalized)
vectors $\zeta \in \mathcal{V}_\sigma $. We then
gain considerable  flexibility and generality in the 
choice of the expanding vector. 

The technical work for the transition from $S$ to $G_\sigma $ is
minimal and consists of modifications of the arguments in the previous
sections.

\begin{lemma}
  \label{lemma:matrixestimates}
  If $\zeta$ and $\varphi$ are smooth vectors for $\pi _\sigma $, there
  is a constant $C_{\varphi,\zeta}$ depending  on $\varphi$
  and $\zeta$ such that
  $$
  |W_\zeta^\sigma(\varphi)(x)| 
  \leq C_{\varphi,\zeta}|W_\psi^\sigma(\psi)(x)|
  \left( 1-\log (1-|x\cdot o|^2)\right) = C _{\varphi , \zeta } \, ( 1-|x\cdot o|^2)^{\sigma
    /2}  \left( 1-\log (1-|x\cdot o|^2)\right)
  $$
  \end{lemma}
\begin{proof}
We recall the  notation
$x= \begin{pmatrix}
    a & b \\
    c^t & d \end{pmatrix}.$
      By Lemma~\ref{bd}(iii)  the functions
$\zeta$ and $\varphi$ are bounded on $\nball$, so
$C_{\varphi , \zeta } = \sup _{w,z\in \nball } |\varphi (w) \zeta (z)|$ is
finite. Furthermore, $$|\pi _\sigma (x) \zeta (z) | = |\bar{d} -
\langle z, b\rangle |^{-\sigma } |\varphi (x ^{-1}\cdot z)| \leq
|d|^{-\sigma} |1 - \langle z, b/d\rangle |^{-\sigma  } \sup _{z\in
  \nball} |\varphi (z)|,$$ therefore
  \begin{align*}
    |W_\zeta^\sigma(\varphi)(x)| 
    &= \left| 
      \int_{\nball} \varphi(z)\overline{\pi_\sigma(x)\zeta(z)} 
      (1-|z|^2)^{\sigma -n -1}\,dz
    \right| \\
    &\leq  C_{\varphi,\zeta} 
    \frac{1}{|d|^\sigma} \int_{\nball} \frac{1}{|1-\dup{z}{x\cdot o}|^\sigma}
    (1-|z|^2)^{\sigma -n -1}\,dz
  \end{align*}
  According to Theorem 1.12(2) in \cite{Zhu2005} this integral
  is comparable to $\log\frac{1}{1-|x\cdot o|^2}$ as $|x\cdot o|\to 1^-$.
Since $W_\psi (\psi ) (x) = d^{-\sigma}= (1 - |x\cdot o|^2) ^{\sigma /2}$,  the desired estimate follows.
\end{proof}

Recall, that if $\sigma$ is rational, then $G_\sigma$ is a finite covering of $\mathrm{SU}(n,1)$. Hence
the inverse image of $K$ is compact and the center of $G_\sigma$ is finite.

\begin{proposition} \label{prop:315}
  Let $\alpha > -1$ and $1\leq p< \infty$ be given and choose a
  rational $\sigma>n$ such that $-1<\alpha<p(\sigma-n)-1$.
 Assume that $\zeta$ and $\varphi$ are smooth vectors for $\pi _\sigma $.
 Then
  \begin{enumerate}
  \item $\zeta$ is $G_\sigma$-cyclic in $\mathcal{V}_\sigma^\infty$ and
    $W_\zeta^\sigma(\varphi)\in L^p_t(G_\sigma)$ for $t+p\sigma/2 >n$.
    \label{prop:315-1}
  \item For $f\in \mathcal{V}_\sigma^{-\infty}$ there is a constant $C$ such that
    $$
    W_\zeta^\sigma (f)*W_\zeta^\sigma (\zeta) = CW_\zeta^\sigma (f)
    $$
    with convolution on $G_\sigma$.    In fact, $C = \Delta_\sigma \|\zeta
    \|^2_{A_{\sigma -n-1}^2}$ where $\Delta _\sigma $
    is the formal
    dimension of $\pi _\sigma $. 
    \label{prop:315-2}
  \item The mapping
    $$
    L^p_{\alpha+n+1-\sigma p/2}(G_\sigma)\times \mathcal{V}_\sigma^\infty \ni 
    (f,\varphi) \mapsto \int_{G_\sigma} f(x)\dup{\pi^*(x)\zeta}{\varphi}\,dx \in \mathbb{C}
    $$
    is continuous.
    \label{prop:315-3}
  \item $f\in \mathcal{A}^p_\alpha$ if and only if
    $f\in \mathcal{V}_\sigma^{-\infty}$ and $W_\zeta^\sigma(f) \in L^p_{\alpha+n+1-\sigma p/2}(G_\sigma)$.
    \label{prop:315-4}
  \item Up to normalization of $\zeta$,  $\mathcal{A}^p_\alpha$ is equal to
    the coorbit $\Co_{\mathcal{V}_\sigma^{\infty}}^\zeta L^p_{\alpha+n+1-\sigma p/2}(G_\sigma)$.
    \label{prop:315-5}
  \end{enumerate}
\end{proposition}

\begin{proof} (\ref{prop:315-1}) Since
  $\pi_\sigma$ is irreducible, every non-zero vector is cyclic.

Next we prove that $W_\zeta^\sigma(\varphi )\in L^p_t(G_\sigma)$ for
$t+p\sigma/2 >n$. Since $x \to (1-|x\cdot o|^2)^\epsilon \log
(1-|x\cdot o|^2)$ is bounded on $G _\sigma $ for every $\epsilon >0$,
we may choose $\epsilon $ and $C>0$, such that $t+p(\sigma -\epsilon
)/2 > n$ and (by Lemma~\ref{lemma:matrixestimates})
$$
|W_\zeta^\sigma(\varphi)(x)| 
  \leq   C _{\varphi , \zeta } \, ( 1-|x\cdot o|^2)^{\sigma
    /2}  \left( 1-\log (1-|x\cdot o|^2)\right) \leq C ( 1-|x\cdot
  o|^2)^{(\sigma - \epsilon ) /2} \, .
$$
  The weight $x\mapsto (1-|x\cdot o|^2) $ is 
  $K$-right-invariant on $G_\sigma $, therefore the decomposition of
  the Haar measure on $G _\sigma $ described in Theorem~\ref{th:2.3} yields 
    \begin{align*}
      \int_{G_\sigma} &|W_\zeta^\sigma(\varphi)(x)|^p(1-|x\cdot o|^2)^t\,dx \\
      &\leq C \int_{G_\sigma}  (1-|x\cdot o|^2)^{\frac{(\sigma
          -\epsilon)p}{2} \, } (1-|x\cdot o|^2)^t\,dx \\
      &= C  \int_{S} \int_{K_\sigma}  (1-|sk\cdot o|^2)^{\frac{(\sigma
          -\epsilon)p}{2} +t}\,dk\,ds \\
      &= C \int_{S} (1-|s\cdot o|^2)^{\frac{\sigma
          -\epsilon}{2} \, p+t}\,\,ds \, ,
    \end{align*}
and the last integral is finite by
Proposition~\ref{prop:integrability}. Thus $W_\zeta^\sigma(\varphi )\in L^p_t(G_\sigma)$ for
$t+p\sigma/2 >n$.

    
 \noindent
(\ref{prop:315-2}) The reproducing formula follows from the orthogonality
    relations for representation coefficients.
    As shown in Proposition 3.2
    in \cite{Christensen2009}, the reproducing formula extends from the Hilbert space
    to $f \in \mathcal{V}_\sigma^{-\infty}$.    
   
  \noindent 
 (\ref{prop:315-3}) 
    As in the proof of Proposition~\ref{prop:38} it suffices to show
    that $W^\sigma _\zeta (\varphi )$ 
    is in     $L^q_{\sigma q/2 -(\alpha+n+1)q/p}(G_\sigma )$ where
    $1/p+1/q=1$. By item (\ref{prop:315-1}) this is the case, if $\frac{\sigma q}{2}
    -\frac{(\alpha+n+1)q}{p} + \frac{\sigma q }{2} > n$, which amounts
    exactly to our assumption that $\alpha < p(\sigma -n) -1$. 


  \noindent
  (\ref{prop:315-4}) and (\ref{prop:315-5}) We already know that this statement is correct
    for $\psi=1_\nball$ in the case of the group $S$
    and therefore also for $G_\sigma$.
    In order to finish the proof we show that the coorbit space
    does not depend on the analyzing vector. To do this we 
    verify that the conditions
    in Theorem 2.7 in \cite{Christensen2011} are satisfied.
     
    The square integrability of $\pi _\sigma $ and unimodularity of
    $G_\sigma$ implies that
    \begin{equation*}
      \int_{G_\sigma} 
      \dup{f}{\pi_\sigma(x)\zeta} \dup{\pi_\sigma(x)\varphi}{\psi} \,dx
      = \frac{1}{\Delta_\sigma} \dup{\varphi}{\zeta}\dup{f}{\psi}
    \end{equation*}
    for $\zeta,\varphi,\psi, f\in \cV_\sigma$.
    If $\zeta,\varphi,\psi$ are smooth then the vector
    \[ \eta = \int_{G_\sigma} 
    \dup{\pi_\sigma(x)\varphi}{\psi}\pi_\sigma(x)\zeta \,dx 
    = \dup{\varphi}{\zeta} \psi\]
    is also smooth (just a multiple of $\psi$), and therefore 
    the mapping 
      \begin{equation*}
        f\mapsto 
        \int_{G_\sigma} 
        \dup{f}{\pi_\sigma(x)\zeta} \dup{\pi_\sigma(x)\varphi}{\psi} \,dx
        = \dup{f}{\eta}
      \end{equation*}
    is weakly continuous from $\mathcal{V}_\sigma^{-\infty}$ to $\mathbb{C}$.
    
    By Corollary~\ref{corc1} the convolutions by both $W_\psi(\zeta)$ 
    and $W_\zeta(\psi)$ on $L^p_{\alpha+n+1-\sigma p/2}(G_\sigma)$ are continuous, 
    and thus the conditions of
    \cite[Theorem 2.7]{Christensen2011} have been verified.
 \end{proof}

\subsection{Atomic decomposition via the group $G_\sigma$}
%

We will now show that every smooth vector, and in particular
every polynomial, provides atomic decompositions and 
frames for the Bergman space.
To our knowledge this is a new result in the  analysis of
several complex variables and goes beyond the atomic decompositions of
Coifman and Rochberg~\cite{Coifman1980}.  In dimension $1$ this consequence of
coorbit theory was observed in~\cite{pap3}. 


\begin{theorem} \label{thm:full}
Assume that $\sigma > n$ is rational, $1\leq p < \infty $,  and $-1<\alpha < p(\sigma-n)-1$.
Fix a non-zero smooth vector $\zeta$  for $\pi _\sigma $. 

(A) Then
$\mathcal{A}_\alpha ^p$ is a coorbit space with respect to the cyclic
vector $\zeta$, precisely, 
$\mathcal{A}^p_\alpha 
= \Co_{\mathcal{V}_\sigma^\infty}^\zeta L^p_{\alpha+n+1-\sigma
  p/2}(G_\sigma) $. In particular, the coorbit space
$\Co_{\mathcal{V}_\sigma^\infty}^\zeta L^p_{\alpha+n+1-\sigma 
  p/2}(G_\sigma) $ is independent of
the cyclic vector $\zeta \in V_\sigma ^\infty $. 

(B) For every smooth vector $\zeta  $  there exists a neighborhood
$U_\epsilon \subseteq G_\sigma$, such that  every
relatively separated and $U_\epsilon$-dense set $X= \{x_i\mid  i\in
I\}\subseteq G_\sigma $ generates
an atomic decomposition  $\{\pi (x_i)\zeta  \mid i\in I\}$ for
$\mathcal{A}^p_\alpha $. Thus every $f\in \mathcal{A}^p_\alpha  $
possesses an unconditionally convergent expansion
$$f = \sum _{i\in I} c_i \pi _\sigma (x_i) \zeta
$$ 
with coefficients in $\ell ^p_{\alpha +n+1-\sigma p/2}(X)$.

(C)  For every smooth vector $\zeta  $  there exists a neighborhood
$U_\epsilon \subseteq G_\sigma$, such that every
relatively separated and $U_\epsilon$-dense set $X= \{x_i \mid  i\in
I\}\subseteq G_\sigma $ generates a Banach frame  $\{ \pi _\sigma (x_i) \zeta \mid i \in I
\}$  for $\mathcal{A}^p_\alpha $. In particular, for
all $f \in \mathcal{A}^p_\alpha  $
$$
A \|f\|_{\mathcal{A}^p_\alpha } \leq \sum _{i\in I} |\langle f, \pi
_\sigma (x_i)\zeta  \rangle |^p (1-|x_i \cdot o|^2)^{\alpha +n+1-\sigma p/2}  \leq B
\|f\|_{\mathcal{A}^p_\alpha } \qquad \text{ for all } f\in
\mathcal{A}^p_\alpha  \, .
$$
\end{theorem}

\begin{proof}
Again we apply Theorem 
~\ref{at} to the representation
$\pi _\sigma $ of the group $G_\sigma $ (instead of $S$) and the
vector $\zeta $ (instead of $\psi $).  In Proposition~\ref{prop:315}
we have already verified the assumptions that guarantee that 
$\Co_{\mathcal{V}_\sigma^\infty}^\zeta L^p_{\alpha+n+1-\sigma p/2}(G_\sigma)$ is well defined.

For the atomic decompositions and Banach frames we
only need to verify the additional property that the convolution
operator  $f \mapsto f \ast |W_\zeta^\sigma(\varphi)|$  is continuous on
$L^p_{\alpha+n+1-\sigma p/2}(G_\sigma)$  for every smooth vector $\zeta$
and $\varphi$. 
Since 
$|W_\zeta^\sigma(\varphi)(x)|\leq C |W_\psi^\sigma(\psi)(x)| 
\left( 1- \log(1-|x\cdot o|^2)\right)$
and the latter function is $K$-bi-invariant, the convolution
operator  reduces
to an integral over the group $S$ and by Lemma~\ref{conv} to an
integral over $\nball $.  
We may  therefore assume that $\widetilde{f}$
is in $L^p_{\alpha+n+1-\sigma p/2}(S)$ and  corresponds
to a function $f\in L^p_{\alpha-\sigma p/2}(\nball)$.
The only difference from
the  calculation on $S$ in \eqref{eq:c6}  is the factor
$-\log(1-|x\cdot o|^2)$ arising from
Lemma~\ref{lemma:matrixestimates}. Since  $|W_\psi^\sigma(\psi)(x)|
= (1-|x\cdot o|^2)^{\sigma /2}$, we obtain 
for every  $\epsilon>0$ (to be chosen later) 
\begin{align*}
  \int_S &|\widetilde{f}(x)| |W_\psi^\sigma(\psi)(x^{-1}y)| 
  \left(1-\log(1-|x^{-1}y\cdot o|^2)\right) \,dx \\
  &\leq C \int_S |\widetilde{f}(x)| (1-|x^{-1}y\cdot o|^2)^{\sigma /2-\epsilon}\,dx \\
  &= C (1-|z|^2)^{(\sigma-\epsilon)/2}
  \int_{\nball} |f(w)| 
  \frac{(1-|w|^2)^{(\sigma-\epsilon)/2-n-1}}{|1-\dup{z}{w}|^{\sigma-\epsilon}}\,dv(w)
\end{align*}
This corresponds to the operator $\mathbf{S}$ in Theorem~\ref{Thm:Zhu1}
with $a=(\sigma-\epsilon)/2$ and
$b=(\sigma-\epsilon)/2-n-1$. According to Theorem~\ref{Thm:Zhu1} the operator
$\mathbf{S}$  is continuous on $L^p_{\alpha-\sigma p/2}(\nball)$ if and only
if 
$-(\sigma-\epsilon)p/2 < \alpha -\sigma p/2+1<p[(\sigma-\epsilon)/2-n]$.
Since the assumption  $-1<\alpha < p(\sigma-n)-1$ is equivalent to
the strict   inequalities 
$-\sigma p/2 < \alpha -\sigma p/2+1< p(\sigma /2-n)$,  a suitable  $\epsilon$ can always be chosen.
\end{proof}

\begin{remark}
  We would like to point out that the group $N$ is isomorphic
  to the Heisenberg group and that the group $A$ acts by dilations.
  Hence one can use the coorbit theory to study Besov spaces
  and their atomic decompositions on the unit ball 
  \cite[Chapter 6]{Zhu2005} using the theory in
  \cite{Christensen2012a,Fuhr2012}.
\end{remark}

\subsection{Further Remarks}

In the case of integrable representations, i.e., $\sigma >2n$ and
$\sigma \in \mathbb{Q}$, the original coorbit theory
in~\cite{Feichtinger1988,Feichtinger1989a,Grochenig1991} can be
applied directly and many of the technical details can be omitted. For
instance, if $W^\sigma _\zeta (\psi ) \in L^1(G_\sigma )$, then the
convolution operator $f \to f \ast |W^\sigma _\zeta (\psi )| $ is
bounded on $L^p(G_\sigma )$ and the subtleties of
Theorem~\ref{Thm:Zhu1}  are not needed. In this case the set of atoms
used in Theorem~\ref{thm:full} can be significantly enlarged. For
simplicity, we discuss  this effect only for coorbits of the unweighted $L^p$-spaces
on $G_\sigma $, which are $\Co _{\mathcal{V}_\sigma}^\zeta L^p
  (G_\sigma ) = \mathcal{A}^p_{\sigma p /2 -  n - 1}$. If $\sigma >2n,
  \sigma \in \mathbb{Q}$, then the definition of $ \Co _{\mathcal{V}_\sigma}^\zeta L^p
  (G_\sigma )$ is independent of $\zeta \in \Co _{\mathcal{V}_\sigma}^\zeta L^1
  (G_\sigma ) = \mathcal{A}^1_{\sigma /2 -  n - 1}$ rather than just
  $\zeta\in \mathcal{V}^\infty _\sigma
  $~\cite[Thm.~4.2]{Feichtinger1989a}. Likewise, in
  Theorem~\ref{thm:full} the atom $\zeta $ for the atomic
  decompositions and frames can be chosen in a space  that is
  strictly between $\mathcal{V}_\sigma ^\infty$ and
  $\mathcal{A}^1_{\sigma /2 -  n - 1}$ (denoted by $\mathcal{B}_w$
  in~\cite[Thm.~4.2]{Feichtinger1989a}).  

As a further benefit of the use of integrable representations we
mention the construction of Riesz sequences in certain Bergman spaces. 
Recall that for a given compact set $K\subseteq G_\sigma $ a set
$X= \{x_i\}_{i\in I}$  is called $K$-separated, if the translates $x_i K$
are pairwise disjoint. 

\begin{proposition}
  \label{riesz}
Assume that $\sigma >2n$, $\sigma \in \mathbb{Q}$, $1< p< \infty $,
and $\zeta \in \mathcal{V}_\sigma ^\infty $. 
Then there exists a
compact set $K \subseteq G_\sigma $ such that every $K$-separated set
$X=\{x_i \}_{i\in I}  \subseteq  G_\sigma $ generates a $p$-Riesz sequence for $\mathcal{A}^p _{\sigma p/2 - n - 1}$.
This means that there exist constants $A,B >0$, such that 
\begin{equation}
  \label{eq:137}
A \|c\|_{\ell ^p(X)} \leq \| \sum _{i\in I} c_i \pi _\sigma (x_i) \zeta
  \|_{\mathcal{A}^p _{\sigma p/2 - n - 1}} \leq B \|c\|_{\ell ^p(I) }
\end{equation}
for all $c \in \ell ^p(X)$. 
\end{proposition}
\begin{proof}
By \cite[Thm.~8.2]{fg89mh} the wavelet transform $W^\sigma _\zeta $ is
interpolating on $ \mathcal{A}^q _{q\sigma /2 -n-1}  = \Co ^\zeta
_{\mathcal{V}_\sigma } L^q(G_\sigma )$ in the sense that there exist a compact set $K\subseteq
G_\sigma $ and a constant $C>0$ with the following property: If $X =
\{x_i\}_{i\in I}  $ is $K$-separated and  $d \in \ell ^q (X)$, then
there exists a function $f \in \mathcal{A}^q _{q\sigma /2 -n-1}$, such
that
$$
(f, \pi _\sigma (x_i) = W^\sigma _\zeta (f)(x_i)  = d_i \quad \forall i \qquad \text{ and }
\qquad 
\|f\| _{\mathcal{A}^q _{q\sigma /2 -n-1}} \leq C \|d\|_q \, .
$$
Now let $c\in \ell ^p (X)$ and  $q$ the conjugate index. By duality
choose $d\in \ell ^q (X)$ with $\|d\|_q = 1$ and $\sum _{i\in I} d_i
\bar{c_i} = \|c\|_p$. According to the interpolation property there is
an $f \in \mathcal{A}^q _{q\sigma /2 -n-1}$ such that $( f, \pi
_\sigma (x_i) \zeta ) = W^\sigma _\zeta
(f)(x_i)   = d_i$ and $\|f\| _{\mathcal{A}^q _{q\sigma /2 -n-1}} \leq
C \|d\|_q = C$. 
  Since $\mathcal{A}^q _{\sigma q/2 - n - 1} = \Co ^\zeta _{V_\sigma
    ^\infty } L^q(G_\sigma )$ is the dual space of  $\mathcal{A}^p _{\sigma p/2 - n - 1} = \Co ^\zeta _{V_\sigma
    ^\infty } L^p(G_\sigma )$ for $1< p <\infty $ (with modifications
  for $p=1,\infty $),  we obtain
  \begin{align*}
    \|c\|_p &= \sum _{i\in I} d_i \bar{c_i} = (f, \sum _{i\in I} c_i
    \pi _\sigma (x_i)\zeta )  \\
&\leq \|f\|_{\mathcal{A}^q _{\sigma q/2 - n - 1}} \|\sum _{i\in I} c_i
    \pi _\sigma (x_i)\zeta  \|_{\mathcal{A}^p _{\sigma p/2 - n - 1}}
    \\
& \leq C  \|\sum _{i\in I} c_i
    \pi _\sigma (x_i)\zeta  \|_{\mathcal{A}^p _{\sigma p/2 - n - 1}}
    \, . 
  \end{align*}
The upper bound in \eqref{eq:137}  always holds  by  Theorem~\ref{at}.
\end{proof}

To summarize, for the Bergman spaces with large $\alpha $ the
condition $\alpha < p (\sigma -n)-1$ forces $\sigma > 2n$ and thus  the
integrability of the representation $\pi _\sigma $. Consequently, for
``most'' Bergman spaces one may use the original coorbit space theory
for integrable representations. However, it is
exactly for small values of $\alpha $ when the extension of coorbit
theory to non-integrable representations
in~\cite{Christensen2009,Christensen2011,Christensen2012} becomes
absolutely essential for the understanding of the  Bergman spaces.

\bibliographystyle{plain}

\end{document}